\documentclass[12pt]{article}

\usepackage[left=1in,right=1in,top=1in,bottom=1in]{geometry}

\usepackage{subfigure}
\usepackage{graphicx}
\usepackage{amsfonts}
\usepackage{url}
\usepackage{amsmath,amssymb}
\usepackage{enumerate}
\usepackage{amsthm}

\usepackage{mathtools}

\usepackage[colorlinks]{hyperref}

\newtheorem{lemma}{Lemma}
\newtheorem{theorem}{Theorem}

\theoremstyle{definition}
\newtheorem{example}{Example}

\newcommand{\R}{{\mathbb R}}
\newcommand{\Z}{{\mathbb Z}}
\newcommand{\eqdef}{\overset{\text{def}}{=}}
\newcommand{\E}{\mathbb E}

\newcommand{\QNl}{\widehat{Q}^\varepsilon}

\newcommand{\DN}{D^\varepsilon}

\title{Multilevel Monte Carlo for stochastic differential equations with small noise}

\author{
David F. Anderson\thanks{Department of Mathematics, University of
  Wisconsin, Madison, USA.  anderson@math.wisc.edu, grant support from NSF-DMS-1318832 and Army Research Office grant W911NF-14-1-0401.},
\and
Desmond J. Higham\thanks{Department of Mathematics and
 Statistics, University of Strathclyde, UK.
d.j.higham@maths.strath.ac.uk, supported by a Royal Society/Wolfson Research Merit Award.},
  \and
 Yu Sun\thanks{Department of Mathematics, University of
  Wisconsin, Madison, USA.  ysun@math.wisc.edu, grant support from NSF-DMS-1318832.}
}

\begin{document}

\maketitle

\begin{abstract}
	
	We consider the problem of numerically estimating expectations of solutions to stochastic differential equations driven by Brownian motions in the commonly occurring small noise regime.  We consider (i)  standard Monte Carlo methods combined with numerical discretization algorithms tailored to the small noise setting, and (ii) a multilevel Monte Carlo method combined with a standard Euler-Maruyama implementation.   Under the assumptions we make on the underlying model, the multilevel method combined with Euler-Maruyama is often found to  be the most efficient option.  Moreover, under a wide range of scalings the multilevel method is found to give the same asymptotic complexity that would arise in the idealized case where we have access to exact samples of the required distribution at a cost of $O(1)$ per sample.  A key step in our analysis is to analyze the variance between two coupled paths directly, as opposed to their $L^2$ distance.  Careful simulations are provided to illustrate the asymptotic results. 
\end{abstract}

\section{Introduction}
\label{sec:intro}

In many modeling and simulation contexts it has proved useful to 
parametrize the diffusion 
coefficient of a stochastic differential equation (SDE)
and study the small noise case.
In particular, 
diffusion and linear noise approximations to jump processes 
arise 
naturally 
under the \lq\lq thermodynamic limit\rq\rq\ 
in biochemistry and cell biology 
\cite{AndKurtz2011,AK2015,G91,G2000}.
Researchers in econometrics and finance may represent market microstructure noise
as small scale diffusion,
 and the task of calibrating model parameters then gives rise to 
small noise SDE simulations;
 see, for example, 
\cite{CMD15,zhmysaha05}, with a more general overview in 
\cite{PS07}.
In computational fluid dynamics, 
small noise SDEs are used as a means to 
incorporate 
thermal fluctuations  
into traditional
models in the 
\lq \lq weak fluctuation regime\rq\rq\ 
\cite[Section~V]{DSGVD14}.
In several other application areas, 
including 
ecology,
circuit simulation,
microbiology, 
neuroscience and population dynamics, 
\cite{BBRT13,denk2007modelling,MMR02,OTN11,SM13,1236910,SMP10,TRW03},
the small noise limit is of interest 
from the perspective of understanding 
properties of physical models. 
The small noise regime has also been investigated as a means to 
validate 
conclusions drawn from analytical or heuristic arguments, especially with regard to 
long-time stability properties 
\cite{IDW98,PRJ01}.

From the perspective of computer simulation, 
many customized numerical methods have been developed for small noise SDEs
with the aim of improving efficiency by exploiting the structure; see \cite[Chapter~3]{MT04}
for an overview.
In this work, we focus on the problem of numerically estimating expectations 
of solutions to small noise SDEs via Monte Carlo and multilevel Monte Carlo methods.  
In particular,  we show that under a range of scalings the standard Euler--Maruyama method combined with the usual multilevel Monte Carlo method of Giles \cite{Giles2008} yields the same complexity that would arise if we had access to exact samples of the
required distribution at a cost of $O(1)$ per sample. So, in this well-defined setting, 
customized methods are not necessary. 

Let 
$(\Omega,{\mathcal F}, \{{\mathcal F}_t\}_{t\ge 0}, P)$ be a filtered probability space satisfying the usual conditions; i.e.~the filtration is complete and right-continuous. Let $W(t)=(W_1(t),W_2(t),\ldots,W_m(t))$ be an $m$-dimensional
standard Wiener processes under $\{{\mathcal F}_t\}_{t \ge0}$.
Let $\varepsilon\in (0,1)$ be a small parameter and let $D^\varepsilon$ be the solution to the following
It\^o SDE,
\begin{equation}\label{eq:2408957248907}
	D^\varepsilon(t)=D(0)+\int_{0}^{t}\mu(D^\varepsilon(s))ds+\varepsilon \int_{0}^{t}\sigma(D^\varepsilon(s))dW(s),
\end{equation}
where $\mu:\R^d \to \R^d$ and $\sigma:\R^d\to \R^{d\times m } $ are  continuous functions satisfying further assumptions detailed below.  

Let $f:\R^d\to \R$ have bounded first and second partial derivatives and let $T >0 $ be a fixed positive number. We are interested in the problem of numerically estimating $\E[f(D^\varepsilon(T))]$ to an accuracy of $\delta >0$ in the sense of confidence intervals.  In particular, we study the computational complexity required to solve this problem utilizing both (i) standard Monte Carlo methods combined with discretization methods tailored to the small noise setting \cite{MT97, milstein1997numerical},  and (ii) multilevel Monte Carlo methods combined with Euler-Maruyama \cite{Giles2008}.  We will show that in the small noise setting the $L^2$ bounds on the difference between exact and approximate processes that are already in the literature \cite{MT97} do not provide sharp estimates for the variance between two coupled paths; an analogous 
issue was previously addressed in the jump process setting \cite{AHS2014}.  Our main effort is  therefore directed at  analyzing the variance between two coupled paths in the small noise setting.

To get a feel for the best possible result, we note that in the idealized case where realizations of $f(D^\varepsilon(T))$ could be generated with a single numerical calculation, and if $\textsf{Var}(f(D^\varepsilon(T))) = O(\varepsilon^2)$, then the computational complexity of solving the problem via Monte Carlo would be $O(\varepsilon^2\delta^{-2}+1)$, where the ``+1'' recognizes the fact that at least one realization must be produced.  We show in this work that when $\delta \ge e^{-\frac1\varepsilon}$ the multilevel Monte Carlo method of Giles \cite{Giles2008} combined with a standard implementation of Euler-Maruyama solves the problem with a computational complexity of $O(\varepsilon^2\delta^{-2} + \delta^{-1})$, which is the same as in the idealized case when $\delta \le \varepsilon^2$, and is otherwise equal to the complexity of the standard Euler method applied to an ordinary differential equation.  
We will show that when $\delta < e^{-\frac1\varepsilon}$ the multilevel Monte Carlo method combined with Euler-Maruyama solves the problem with a complexity of $O(\varepsilon^4 \delta^{-2} \log(1/\delta)^2)$.  

We also demonstrate below that when $\delta < \varepsilon^2$, methods customized to the small noise setting combined with standard Monte Carlo  can sometimes be more efficient than the multilevel Monte Carlo method combined with standard Euler-Maruyama.  This occurs because in the regime $\delta < \varepsilon^2$ the majority of the required work falls on accurately computing the drift in \eqref{eq:2408957248907}, and not due to the randomness of the process.

We make the following regularity assumption throughout the manuscript.

\vspace{.125in}

\noindent \textbf{Running assumption.}  \textit{We suppose there are constants $a,b>0$ such that for all $x, y \in \R^d$ the following inequalities hold:
\[
  |\nabla\mu(x)|^2\vee|\nabla^2\mu(x)|^2\le a,
\]
and
    \[
     |\mu(x)-\mu(y)|^2\le a|x-y|^2,\quad |\sigma(x)-\sigma(y)|^2\le b|x-y|^2,
\]
and
\[
 |\mu(x)|^2\le a(1+|x|^2),\quad |\sigma(x)|^2\le b(1+|x|^2).
\]
}

\noindent
 Under the above assumptions, the SDE \eqref{eq:2408957248907} is known to have
a unique strong solution (see, for example, Theorem 3.1 on page 51 in \cite{mao1997stochastic}).    We also note that when these assumptions are violated the multilevel Monte Carlo method may fail,
but the performance can be recovered by modifying the Euler--Maruyama discretization \cite{HJK13}.

\vspace{.125in}

\subsection{Euler-Maruyama and a statement of main mathematical result}
We provide a continuous version of the Euler-Maruyama discretization method.
Let $h>0$ and let $D_h^\varepsilon$ be the  solution to
\begin{equation}\label{eq:48907528}
	D_h^\varepsilon(t) = D(0) + \int_0^t \mu(D_h^\varepsilon(\eta_h(s))) \, ds + \varepsilon \int_0^t \sigma(D_h^\varepsilon(\eta_h(s)))\, dW(s),
\end{equation}
where $\eta_h(s) \eqdef \lfloor s/h\rfloor h$, for $s \ge 0$. It is straightforward to see that the solution to \eqref{eq:48907528} restricted to the set of times $\{0,h,2h,\dots\}$ has the same distribution as the discrete time process generated by the usual Euler-Maruyama method \cite{KloedenPlaten92}.

In order to understand the computational complexity of the multilevel scheme, we need sharp estimates for the variance between two coupled paths.  The following provides such an estimate and is the main theorem provided in this paper. 
The result bounds  the variance between two coupled process;  both are generated via \eqref{eq:48907528}, though they have different time discretization parameters.  See the beginning of section \ref{sec:res} for more details related to the coupling.

\begin{theorem}
  \label{thm:var}
		   Suppose the functions $\mu$ and $\sigma$ satisfy our running assumptions and that $T>0$ and $\varepsilon\in(0,1)$.  Suppose further that $D_{h_\ell}^\varepsilon(t)$ and $D_{h_{\ell-1}}^\varepsilon(t)$ satisfy \eqref{eq:48907528} with time discretization parameters $h_\ell=T\cdot M^{-\ell}$ and $h_{\ell-1}=T\cdot M^{-(\ell-1)}$, respectively, where $M\ge 2$ is a positive integer, and that these two processes are constructed with the same realization of Brownian motions.  Assume that $f:\R^d \to \R$ has continuous second derivative and there exists a constant $C_L$ such that
     \[
       \left\| \frac{\partial f}{\partial x_i} \right\|_\infty \leq C_L \quad and \quad \left\|\frac{\partial^2 f}{\partial x_i\partial x_j}\right\|_\infty\leq C_L\quad for \,\,any\,\,i,j=1,2,...,d.
     \]
	Then, for $\ell \ge 1$,
		\begin{equation}
		\label{eq:main}
			\max_{0\le n\le  M^{\ell-1}}\textsf{Var}(f(D_{h_{\ell}}^\varepsilon(t_n)) - f(D_{h_{\ell-1}}^\varepsilon(t_n)))\leq \bar{C}_1h_{\ell-1}^2\varepsilon^2+\bar{C}_2h_{\ell-1}\varepsilon^4,
		\end{equation}
where $t_n = n\cdot h_{\ell-1}$, and  $\bar{C}_1$ and $\bar{C}_2$ are positive constants only depending on $a,b,d,m,T,D(0)$ and $C_L$.
\end{theorem}

In the context of analyzing the classical mean-square error,   
it was shown by Milstein and Tretyakov in \cite{MT97} that under the same assumptions as in Theorem \ref{thm:var}, 
\begin{equation}\label{eq:MT}
  	 \E[|f(D^\varepsilon(T)) - f(D^\varepsilon_{h}(T))|^2] = O(h^2 + h \varepsilon^4), 
\end{equation}
  where $D^\varepsilon$ is the solution to \eqref{eq:2408957248907}. 
  We note that the $O(h^2)$ term cannot be avoided when we analyze the mean-square error because the underlying deterministic Euler method is first order. From the mean-square error bound \eqref{eq:MT} we can deduce that for some $C_1, C_2 > 0$, we have $\max_{0\le n\le  M^{\ell-1}}\textsf{Var}(f(D_{h_{\ell}}^\varepsilon(t_n)) - f(D_{h_{\ell-1}}^\varepsilon(t_n)))\leq {C}_1h_{\ell-1}^2+{C}_2h_{\ell-1}\varepsilon^4$, where, again, $t_n = n\cdot h_{\ell-1}$.
 Theorem \ref{thm:var} sharpens this bound considerably, showing that the 
overall variance scales favorably with $\varepsilon$, even though the
Euler--Maruyama method 
has not been customized to exploit the small noise property.

\section{Complexity analysis}
\subsection{Standard Monte Carlo methods}\label{sec:2.1}
As a basis for comparison, we first analyze the complexity of standard Monte Carlo
with a general discretization method.   

Suppose $D_h^\varepsilon$ is generated by a numerical scheme (not necessarily \eqref{eq:48907528}) for which the bias of the discretization method satisfies 
\begin{equation}\label{eq:908897}
	|\E [f(D_{h}^\varepsilon(T))]-\E [f(D^\varepsilon(T))]|=O(h^{p}+\varepsilon^{r}h^{q}),
\end{equation}
where $q\le p$ and $r \ge 0$ (see \cite{milstein1997numerical}, where some such methods are provided). In order to ensure that the bias \eqref{eq:908897} is of order $\delta$, we require that
\begin{equation}\label{eq:67697696}
h = O(\min(\delta^{1/p},\delta^{1/q}\varepsilon^{-r/q})).
\end{equation}
Under our running assumptions and assuming Lemma \ref{lem:var_bound} below, which applies to Euler-Maruyama, holds for these customized methods we find 
$$\textsf{Var}(f(D_h^\varepsilon(T))) =\textsf{Var}(f(D_h^\varepsilon(T))-f(z_h(T))) \le C\E \left[\sup_{s \le T}	 |D_{h}^\varepsilon(s) - z_h(s)|^2\right]= O(\varepsilon^2),$$ 
where $z_h$ is the Euler solution to the associated deterministic model obtained when $\varepsilon$ is set to $0$ in \eqref{eq:2408957248907}, see \eqref{gen_approx_ode}. 
Thus, the standard Monte Carlo estimator
\[
	\E[f(D^\varepsilon(T))] \approx \E[f(D_h^\varepsilon(T))] \approx \frac{1}{n} \sum_{i = 1}^n f(D_{h,[i]}^\varepsilon(T)),
\]
where $D_{h,[i]}^\varepsilon$ is the $i$th independent realization of the process $D_h^\varepsilon$, 
has a variance that is $O(n^{-1} \varepsilon^2)$.  To produce an overall estimator variance of $O(\delta^2)$, we require that $n = O(\varepsilon^2 \delta^{-2}+1)$, where the ``$+1$'' captures the requirement that at least one path must be generated.  Assuming  that the cost of generating a single path of the scheme scales like $h^{-1}$, we
obtain an upper bound on the overall computational complexity of order 
\begin{equation}\label{eq:578965}
O((\varepsilon^2 \delta^{-2} +1)h^{-1})=O\left(\frac{\varepsilon^2\delta^{-2} + 1}{\min(\delta^{1/p},\delta^{1/q}\varepsilon^{-r/q})}\right).
\end{equation}

For example, with the Euler-Maruyama scheme \eqref{eq:48907528} we have that $p = q= 1, r=0$ yielding a bias of $O(h)$ in \eqref{eq:908897}. 
In this case we  select $h = O(\delta)$, and find a computational complexity of $O(\varepsilon^2 \delta^{-3} + \delta^{-1})$.  

To see how a customized method may be beneficial, consider the case $\delta = O(\varepsilon^{\rho})$ for $\rho \in (1,2)$.  We may select a method with $p=2, r = 2,q = 1$ (see section 5 of \cite{milstein1997numerical}) in which case \eqref{eq:67697696} gives $h = O(\varepsilon^{\rho/2})$, and \eqref{eq:578965} yields a computational complexity of order $O(\varepsilon^{2-\frac52\rho})$; see subsection \ref{sec:comparisons} for further details..

\subsection{Euler-based multilevel Monte Carlo}
\label{subsec:emmc}
Here we specify and analyze an Euler-Maruyama based multilevel Monte Carlo method for the diffusion approximation.  We follow the original framework of Giles \cite{Giles2008}.

For a fixed positive integer $M\ge 2$ we let $h_{\ell} = T \cdot M^{-\ell}$ for $\ell \in \{0,\dots,L\}$.  Reasonable choices for $M$ include $M\in \{2,3,4,5,6 , 7\}$, and $L$ is determined below.
For each $\ell \in \{0,1,\dots, L\}$, let $\DN_{h_\ell}$ denote the approximate process generated by \eqref{eq:48907528} with a step size of $h_{\ell}$.
Note that
\begin{align*}
	\E [f(D^\varepsilon(T))] &\approx \E [f(D_{h_L}^\varepsilon(T))] =  \E[f(D^\varepsilon_{h_0}(T))] +  \sum_{\ell = 1}^L \E[ f(D^\varepsilon_{h_{\ell}}(T)) - f(D^\varepsilon_{h_{\ell-1}}(T))],
\end{align*}
with the quality of the approximation only depending upon $h_L$.  As mentioned in \cite{milstein1997numerical}, the Euler discretization has a weak order of one in the present setting for a large class of functionals $f$.  Hence, we set $h_L = \delta$ in order for the bias to be $O(\delta)$.  This choice yields $L = O( \log(1/\delta) )$.
We now let
\begin{align*}
	\QNl_{0} &\eqdef \frac{1}{n_{0}} \sum_{i = 1}^{n_{0}} f(\DN_{h_0,[i]}(T)), \quad \text{and} \quad
	\QNl_{\ell} \eqdef \frac{1}{n_{\ell}} \sum_{i = 1}^{n_{\ell}} ( f(\DN_{h_\ell,[i]}(T)) - f(\DN_{h_{\ell-1},[i]}(T))),
\end{align*}
for $\ell = 1, \dots, L$, where $n_0$ and the different $n_\ell$ have yet to be determined.  
Our  estimator is then
\[
  \QNl \ \eqdef \
  \QNl_0 +
  \sum_{\ell = 1}^L \QNl_{\ell},
\]
which is the usual multilevel Monte Carlo estimator \cite{Giles2008}.
 Set
\[
	\delta_{\varepsilon,\ell} = \textsf{Var}  ( f(\DN_{h_\ell}(T)) - f(\DN_{h_{\ell-1}}(T))).
\]
 By Theorem \ref{thm:var}, we have $\delta_{\varepsilon,\ell}=O(h_\ell^2\varepsilon^2+h_\ell \varepsilon^4)$ under a wide array of circumstances.  Also note that $\delta_{\varepsilon,0}=\textsf{Var}(f(D_{0}^\varepsilon)) = O(\varepsilon^2).$ 
 
 For $\ell \in\{1,\dots,L\}$, let $C_\ell$  be the computational complexity required to generate a single pair of coupled trajectories at level $\ell$.  Let $C_0$  be the computational complexity required to generate a single trajectory at the coarsest level.  To be concrete, we set $C_\ell$ to be the number
of random variables required to generate the requisite path.  To determine $n_{\ell}$,
 we solve the following optimization problem, which ensures a total variance of $\QNl$ no greater than order $\delta^2$:
 \begin{align}
&\underset{n_{\ell}}{\text{minimize}} \hspace{.15in} \sum_{\ell = 0}^L n_{\ell}C_{\ell},\label{eq:LM1}\\
&\text{subject to} \hspace{.1in} \sum_{\ell = 0}^L \frac{\delta_{\varepsilon,\ell}}{n_{\ell}}= \delta^2.\label{eq:LM2}
\end{align}
     We use Lagrange multipliers.  Since $C_{\ell}=K\cdot h_{\ell}^{-1}$, for some fixed constant $K$, the optimization problem above is solved at solutions to
\begin{equation*}
 \nabla_{n_0,\dots,n_L,\lambda}\left(
 \sum_{\ell = 0}^L n_{\ell}K\cdot h_{\ell}^{-1}+\lambda\left(\sum_{\ell = 0}^L \frac{\delta_{\varepsilon,\ell}}{n_{\ell}}- \delta^2\right)\right) = 0.
\end{equation*}
By taking a derivative with respect to $n_{\ell}$ we obtain,
\begin{equation}
\label{eq:nl}
n_{\ell}= \sqrt{\tfrac{\lambda}{K}\delta_{\varepsilon,\ell}h_{\ell}}, \qquad \text{ for } \ell \in\{0,1,2,\dots,L\}
\end{equation}
for some $\lambda\ge 0$. Plugging \eqref{eq:nl} into \eqref{eq:LM2} yields
\begin{equation}\label{eq:78966798}
\sum_{\ell = 0}^L\sqrt{ \frac{\delta_{\varepsilon,\ell}}{h_{\ell}}}= \sqrt{\tfrac{\lambda}{K}}\cdot \delta^2
\end{equation}
and hence, by Theorem \ref{thm:var} there is a $C>0$ for which
\begin{equation}\label{eq:09877890}
 \sqrt{\tfrac{\lambda}{K}}=\delta^{-2}\sum_{\ell = 0}^L\sqrt{ \frac{\delta_{\varepsilon,\ell}}{h_{\ell}}} \le C\delta^{-2}\sum_{\ell = 0}^L\sqrt{ h_\ell\varepsilon^2+\varepsilon^4} \le \tilde C\delta^{-2}(\varepsilon+\varepsilon^{2} L),
\end{equation}
where in the final inequality we used that $\sqrt{a+b} \le \sqrt{a} + \sqrt{b}$ for non-negative $a,b$, and $\tilde C$ is a new constant.
Recall that $L=O(\log(1/\delta))$.  Hence, if $\delta\ge e^{-\frac{1}{\varepsilon}}$, which is equivalent to $\varepsilon^2 L \le \varepsilon$, then
\[
\frac{\lambda}{K}=O(\delta^{-4}\varepsilon^2).
\]
Plugging this back into \eqref{eq:nl}, and recognizing that we must have $n_\ell \ge 1$, yields
\[
	n_{\ell} = O(\delta^{-2} \varepsilon \sqrt{\delta_{\varepsilon,\ell} h_\ell} + 1).
\]
Hence, we see that in this case of $\delta \ge e^{-\frac{1}{\varepsilon}}$ the overall computational complexity is
\begin{align}
\sum_{\ell=0}^L n_\ell K \cdot h_\ell^{-1} = O \left( \sum_{\ell=0}^L \delta^{-2} \varepsilon \sqrt{\delta_{\varepsilon,\ell}h_{\ell}^{-1}} + \sum_{\ell=0}^L h_{\ell}^{-1}\right) = O \left(  \varepsilon^2\delta^{-2}  +  \delta^{-1}\right).
\label{eq:reg1}
\end{align}

%

If $\delta < e^{-\frac{1}{\varepsilon}}$, in which case $\varepsilon^2 L > \varepsilon$, then $\frac{\lambda}{K} = O(\delta^{-4} \varepsilon^4 L^2)$, and \eqref{eq:nl} yields
\begin{align*}
	n_\ell = O\left(\varepsilon^2\delta^{-2} L \sqrt{\delta_{\varepsilon,\ell} h_\ell}\right).
\end{align*}
Since $\delta^{-2}\varepsilon^2 L \sqrt{\delta_{\varepsilon,\ell} h_\ell} \ge 0.89$ when $\delta < e^{-\frac1\varepsilon}$, we see that the usual ``+1'' term is not necessary in the expression above.  
We may now conclude that the overall computational complexity under the assumption $\delta < e^{-\frac{1}{\varepsilon}}$ is
\begin{align}\label{eq:reg2}
	\sum_{\ell=0}^L n_\ell K \cdot h_\ell^{-1} = O( \varepsilon^4\delta^{-2} \log(1/\delta)^2).
\end{align}

\subsection{Comparisons}
\label{sec:comparisons}
There are multiple scaling regimes to consider.  We begin with $\delta < e^{-\frac1\varepsilon}$,  which represents a severe accuracy requirement.  Under this assumption, the computational complexity of multilevel Monte Carlo with Euler-Maruyama is given by  \eqref{eq:reg2}, whereas the  complexity \eqref{eq:578965} required for methods tailored to the small noise setting is
\[
O\left( \varepsilon^2\delta^{-2}  (\delta^{-\frac1p} + \delta^{-\frac1q} \varepsilon^{r/q}) \right).
\]
Hence, so long as 
\begin{align}\label{eq:89769768}
	\varepsilon^2 \log(1/\delta)^2 < \delta^{-\frac1p} + \delta^{-\frac1q} \varepsilon^{r/q} 
\end{align}
 multilevel Monte Carlo combined with Euler-Maruyama is most efficient, and there is no need to utilize customized methods.  To get a sense of the restriction \eqref{eq:89769768}, we note that if $p = 2$, then \eqref{eq:89769768} holds  so long as $\varepsilon < 0.65$, and if $p = 4$, then  \eqref{eq:89769768} holds  so long as $\varepsilon < 0.33$. 
   In fact, under the further assumption that $\delta \approx e^{-\frac1\varepsilon}$ we see that \eqref{eq:reg2} is $O(\varepsilon^2 \delta^{-2})$, the same ---asymptotically in the parameters $\delta$ or $\varepsilon$---as in the situation where we can generate independent realizations of $f(D^\varepsilon(T))$ exactly in a single step.  As  $\delta$ decreases below this threshold, the ratio between \eqref{eq:reg2} and the complexity in the idealized setting considered in the introduction grows like $\log(1/\delta)^2$, as is common in the multilevel setting:
\[
	\frac{\varepsilon^4 \delta^{-2} \log(1/\delta)^2}{\varepsilon^2 \delta^{-2}} = \varepsilon^2 \log(1/\delta)^2.
\]

\vspace{.1in}


Turning to the case $\delta \ge e^{-\frac1\varepsilon}$, there are two relevant subcases to consider.  First, in the regime $\delta \le \varepsilon^{2}$ we have $\varepsilon^2 \delta^{-2} \ge \delta^{-1}$ and the  complexity \eqref{eq:reg1} is of order $O(\varepsilon^{2}\delta^{-2})$.  
This bound compares favorably with the bound  
$O(\varepsilon^2 \delta^{-3})$
that we derived in
subsection~\ref{subsec:emmc}
for standard Monte Carlo with Euler--Maruyama, and allows us to 
carry through 
a conclusion that applies to general SDEs 
\cite{Giles2008}:
multilevel Monte Carlo can improve on the complexity 
of standard Monte Carlo by a factor $\delta^{-1}$, where 
$\delta $ is the required accuracy. 
Moreover, and still under the assumption that $\delta \le \varepsilon^2$, the complexity $O(\varepsilon^2 \delta^{-2})$ is uniformly superior to the complexity \eqref{eq:578965} required for methods tailored to the small noise setting.  Hence, we may conclude that when $\delta \le \varepsilon^2$, there is no need to use such tailored methods.  Finally,  following the discussion in section~\ref{sec:intro} and in the paragraph above, 
we note that this multilevel Euler computational complexity is the same---asymptotically in the parameters $\delta$ or $\varepsilon$---as in the situation where we can generate independent realizations of $f(D^\varepsilon(T))$ exactly in a single step.

The last case to consider  is $\delta > \varepsilon^2$.  Now the complexity \eqref{eq:reg1} is of order $O(\delta^{-1})$, the same as Euler's method applied to an ordinary differential equation.  In this case, well selected customized methods can be asymptotically more efficient than multilevel Monte Carlo combined with standard Euler-Maruyama.  For example, and following the discussion at the end of section \ref{sec:2.1}, if $\delta = \varepsilon^\rho$ for some $\rho \in (1,2)$, then the multilevel method with Euler-Maruyama requires a complexity of order $O(\varepsilon^{-\rho})$.  However, a customized method with $p = 2, r = 2, q=1$ requires a complexity of order $O(\varepsilon^{2-\frac52 \rho})$.  Hence, the customized method is superior when $\rho \in (1,\frac43)$.

Finally, it is tempting to think that the computational complexity of the multilevel scheme found above can be heuristically derived in the following manner.  Start with a continuous time Markov chain model which satisfies a scaling so that \eqref{eq:2408957248907} is a natural diffusion approximation of the jump process.  Next, use the results of \cite{AHS2014}, which are related to the variance between two coupled paths of the jump process, to infer the proper scaling in the diffusive regime.  Somewhat surprisingly, this heuristic does not work and leads to overly pessimistic results.  We delay a deeper discussion of this issue until section \ref{sec:comparisonCTMC}, where we address this issue both analytically and computationally.


\section{Proof of Theorem \ref{thm:var}}
\label{sec:res}

Throughout this section, we assume the conditions of Theorem \ref{thm:var} are met with positive integer $M$ fixed.

The coupling of the two approximate processes, $D^\varepsilon_{h_{\ell}}(t)$ and $D^\varepsilon_{h_{\ell-1}}(t)$, takes the form
\begin{align*}
	D_{h_{\ell}}^\varepsilon(t) &= D(0) + \int_0^t \mu(D_{h_{\ell}}^\varepsilon(\eta_{h_\ell}(s))) \, ds + \varepsilon \int_0^t \sigma(D^\varepsilon_{h_{\ell}}(\eta_{h_\ell}(s)))\, dW(s), \\ 
	D_{h_{\ell-1}}^\varepsilon(t) &= D(0) + \int_0^t \mu(D_{h_{\ell-1}}^\varepsilon(\eta_{h_{\ell-1}}(s))) \, ds + \varepsilon \int_0^t \sigma(D^\varepsilon_{h_{\ell-1}}(\eta_{h_{\ell-1}}(s)))\, dW(s).
\end{align*}
For $n \in \{0,1,\dots,M^{\ell-1}\}$ and $k \in \{0,\dots,M\}$ let 
\[
	 t_n=nh_{\ell-1},\quad \text{and} \quad t_n^k=nh_{\ell-1}+kh_{\ell}.
\]
Note that for each $n$ we have
\[
t_{n}^0=t_n,\quad t_n^M=t_{n+1}.
\]
We use the following discretization scheme to simulate the coupling above.  First, for each $n \in \{0,1,\ldots,M^{\ell-1}\}$ and $k \in \{0,\dots,M-1\}$, let 
\begin{align}
D_{h_{\ell}}^\varepsilon(t_{n}^{k+1}) =D_{h_{\ell}}^\varepsilon(t_{n}^{k}) + \mu(D_{h_{\ell}}^\varepsilon(t_{n}^k)){h_\ell} +  \varepsilon\sqrt{h_\ell}\sigma(D_{h_{\ell}}^\varepsilon(t_{n}^k))W_{n}^{k},
  \label{eq:d11}
\end{align}
where  the random vector $W_{n}^{k} \in \R^m$ has independent components (from each other and all previous random variables), and each component is distributed as $N(0,1)$.   Note that \eqref{eq:d11} implies
\begin{align*}
 	D_{h_{\ell}}^\varepsilon(t_{n+1}) =D_{h_{\ell}}^\varepsilon(t_{n}) + \sum_{k=0}^{M-1} \mu(D_{h_{\ell}}^\varepsilon(t_{n}^k)){h_\ell} + \varepsilon\sqrt{h_\ell}\sum_{k=0}^{M-1} \sigma(D_{h_{\ell}}^\varepsilon(t_{n}^k))W_{n}^{k}.
     \end{align*}
To simulate $D_{h_{\ell-1}}^\varepsilon$, we then use
    \begin{align*}
	D_{h_{\ell-1}}^\varepsilon(t_{n+1}) = D_{h_{\ell-1}}^\varepsilon(t_{n}) + \mu(D_{h_{\ell-1}}^\varepsilon(t_{n})){h_{\ell-1}} + \varepsilon\sqrt{h_{\ell-1}} \sigma(D_{h_{\ell-1}}^\varepsilon(t_{n})) \sum_{k=0}^{M-1} W_{n}^{k}.
\end{align*}

We begin with a series of necessary lemmas.

\begin{lemma}
\label{lem:moment_bound}
	For any $T>0$ we have
\[
			\E\left[\sup_{0\le s\le T} |D_{h_\ell}^\varepsilon(s)|^4\right] \leq  C,
\]
for some $C= C(a,b,T,D(0))$.
\end{lemma}
\begin{proof}
For any $t>0$,
\begin{align*}
   |D_{h_{\ell}}^\varepsilon(t)|^4\le27|D(0)|^4+27\left| \int_{0}^{t}\mu(D_{h_{\ell}}^\varepsilon(\eta_{h_\ell}(s)))ds \right|^4+27\varepsilon^4\left|\int_{0}^{t}\sigma(D_{h_{\ell}}^\varepsilon(\eta_{h_\ell}(s)))dW(s)\right|^4.
\end{align*}
Thus, 
\begin{align}
   \sup_{0\le s \le t}|D_{h_{\ell}}^\varepsilon(s)|^4&\le 27|D(0)|^4+27t^3\int_{0}^{t}\sup_{0\le r \le s}|\mu(D_{h_{\ell}}^\varepsilon(\eta_{h_{\ell}}(r)))|^4ds\notag\\
   &\hspace{.2in}+ 27\varepsilon^4\sup_{0\le s \le t}\left|\int_{0}^{s}\sigma(D_{h_{\ell}}^\varepsilon(\eta_{h_{\ell}}(r)))dW(r)\right|^4,\label{eq:BDG88}
\end{align}
since the right-hand-side is monotonically increasing in $t$.
Applying the Burkholder-Davis-Gundy inequality  \cite{mao1997stochastic} to the  term \eqref{eq:BDG88} and taking expectations we get
\begin{align}
\begin{split}
    \E\left[\sup_{0\le s \le t}|D_{h_{\ell}}^\varepsilon(s)|^4\right] &\le27|D(0)|^4+27t^3\int_{0}^{t}\E\left[\sup_{0\le r \le s}|\mu(D_{h_{\ell}}^\varepsilon(\eta_{h_{\ell}}(r)))|^4\right]ds\\
   &\quad+K(T)\varepsilon^4\int_{0}^{t}\E[|\sigma(D_{h_{\ell}}^\varepsilon(\eta_{h_{\ell}}(s)))|^4]ds,
   \end{split}
   \label{eq:cont_z}
\end{align}
where $K(T)$ is a generic constant only depending on $T$.
Using \eqref{eq:cont_z} with $t = nh_{\ell}$ and $s = mh_{\ell}$, where $n$ and $m$ are nonnegative integers for which $mh_{\ell} \le nh_{\ell}\le t \le T$, we get
\begin{align*}
   \E\left[\sup_{m \le n}|D_{h_{\ell}}^\varepsilon(mh_{\ell})|^4\right]   &\le27|D(0)|^4+27t^3\sum_{i=0}^{n-1}\E\left[ \sup_{m \le i}|\mu(D_{h_{\ell}}^\varepsilon(mh_\ell))|^4\right]h_\ell \\
   &\quad +K(T)\varepsilon^4\sum_{i=0}^{n-1}\E\left[|\sigma(D_{h_{\ell}}^\varepsilon(ih_\ell))|^4\right]h_\ell\\
   &\le 27|D(0)|^4+54a^2T^4+K(T)b^2\varepsilon^4\\
   &\hspace{.2in}+(54a^2T^3+K(T)b^2\varepsilon^4)\sum_{i=0}^{n-1}\E\left[\sup_{m \le i}|D_{h_{\ell}}^\varepsilon(mh_\ell)|^4\right]h_\ell,
\end{align*}
where in the final inequality we applied the growth conditions for both $\mu$ and $\sigma$ found in the running assumption.
We then use the discrete version of Gronwall's Lemma to obtain 
\[
\E\left[\sup_{m \le n}|D_{h_{\ell}}^\varepsilon(mh)|^4\right]\le C_1(a,b,T,D(0)).
\]
Now we return to \eqref{eq:cont_z} and, after  applying the growth conditions pertaining to both $\mu$ and $\sigma$ in our running assumption, conclude
\begin{equation*}
	\E \left[\sup_{0\le s \le T}|D_{h_{\ell}}^\varepsilon(s)|^4\right] \le C(a,b,T,D(0)),
\end{equation*}
for some new constant $C$.
\end{proof}

Let $z_h$ be the deterministic solution to
\begin{equation}\label{gen_approx_ode}
	z_h(t) = D(0) + \int_0^t \mu(z_h(\eta_h(s)))ds,
\end{equation}
which is an Euler approximation to the ODE obtained from \eqref{eq:2408957248907} when $\varepsilon$ is set to zero.

\begin{lemma}\label{lem:var_bound}
	For any $T>0$ we have
\[
		 \E \left[\sup_{0\le s \le T}	 |D_{h_{\ell}}^\varepsilon(s) - z_{h_\ell}(s)|^2\right] \leq  C\varepsilon^2,
\]
for some $C=C(a,b,T,D(0))$.
\end{lemma}
\begin{proof}
For $t\le T$, we have
\begin{align*}
|D_{h_{\ell}}^\varepsilon(t)-z_{h_\ell}(t)|^2&\le 2T\int_0^t|\mu(D_{h_{\ell}}^\varepsilon(\eta_{h_{\ell}}(s)))-\mu(z_{h_\ell}(\eta_{h_{\ell}}(s))|^2ds\\
&\hspace{.2in}+2\varepsilon^2\left|\int_0^t\sigma(D_{h_{\ell}}^\varepsilon(\eta_{h_{\ell}}(s)))dW(s)\right|^2.
\end{align*}
As a result of the Burkholder-Davis-Gundy inequality and our running assumptions,
\begin{align}
\begin{split}
\E &\left[\sup_{0\le s\le t}  \left|D_{h_{\ell}}^\varepsilon(s) - z_{h_\ell}(s)\right|^2\right] \\
& \le 2aT\int_0^t\E\left[ \sup_{0\le s\le r} |D_{h_{\ell}}^\varepsilon(\eta_{h_{\ell}}(s)) - z_{h_\ell}(\eta_{h_{\ell}}(s))|^2\right] dr+8\varepsilon^2\int_0^t\E[|\sigma(D_{h_{\ell}}^\varepsilon(\eta_{h_{\ell}}(s)))|^2]ds\\
& \le 8bT\E\left[\sup_{0\le s\le t}(1+|D_{h_{\ell}}^\varepsilon(s)|^2)\right]\varepsilon^2+2aT\int_0^t \E\left[\sup_{0\le s\le r} |D_{h_{\ell}}^\varepsilon(\eta_{h_{\ell}}(s)) - z_{h_\ell}(\eta_{h_{\ell}}(s))|^2\right] dr.
\end{split}
\label{eq:needed_bound14545}
\end{align}

Specializing the above to $t = nh_{\ell}$ and $s = mh_{\ell}$, where $n$ and $m$ are nonnegative integers for which $mh_{\ell} \le nh_{\ell}\le t \le T$, we get 
 \begin{align*}
\E  \bigg[\sup_{m\le n}& |D_{h_{\ell}}^\varepsilon(mh_\ell) - z_{h_\ell}(mh_\ell)|^2\bigg] \\
&\le 8bT\E\left[\sup_{0\le s\le t}(1+|D_{h_{\ell}}^\varepsilon(s)|^2)\right]\varepsilon^2+ 2aT\sum_{i = 0}^{n-1}h_\ell\cdot  \E\left[\sup_{m\le i } |D_{h_{\ell}}^\varepsilon(mh_\ell) - z_{h_\ell}(mh_\ell)|^2\right]\\
&\le 8bT(1+K)\varepsilon^2+2aT\sum_{i = 0}^{n-1}h_\ell\cdot  \E\left[\sup_{m\le i } |D_{h_{\ell}}^\varepsilon(mh_\ell) - z_{h_\ell}(mh_\ell)|^2\right],
\end{align*}
for some $K=K(a,b,T,D(0))$, where the first inequality follows from \eqref{eq:needed_bound14545} and the second utilizes Lemma \ref{lem:moment_bound}.

By the discrete version of Gronwall's inequality we see
\begin{align*}
\E \left[\sup_{m\le n} |D_{h_{\ell}}^\varepsilon(mh_\ell) - z_{h_\ell}(mh_\ell)|^2 \right]\leq ( 8bT(1+K))e^{2aT^2}\varepsilon^2.
\end{align*}
Since $n$ satisfying $nh_\ell\le T$ was arbitrary, we return to \eqref{eq:needed_bound14545}  to conclude
that for any $0\le t\le T$
\begin{align*}
\E \left[ \sup_{s\le t} |D_{h_{\ell}}^\varepsilon(s) - z_{h_\ell}(s)|^2\right] &\le C(a,b,T,D(0))\varepsilon^2.
\qedhere
		\end{align*}		
\end{proof}

\begin{lemma}\label{lem:local_error1} 
	\[
		 \max_{\substack{0\le n\le M^{\ell-1}\\ 1\le k\le M}}|\E [D_{h_{\ell}}^\varepsilon(t_{n}^k) - D_{h_{\ell}}^\varepsilon(t_{n})]|\le CMh_{\ell},
	\]
where $C$ is a positive constant that only depends on $a,b,T,m,D(0)$.
 \end{lemma}
\begin{proof}
Iterating \eqref{eq:d11} yields
\begin{align*}
 \left|\E  \left[D_{h_{\ell}}^\varepsilon(t_{n}^{k})-D_{h_{\ell}}^\varepsilon(t_{n})\right] \right| &\le \left|\E\left[\sum_{i=0}^{k-1}\mu(D_{h_{\ell}}^\varepsilon(t_{n}^i)){h_\ell}\right]\right| + \left|\E\left[ \varepsilon\sqrt{h_\ell}\sum_{i=0}^{k-1}\sigma(D_{h_{\ell}}^\varepsilon(t_{n}^i))W_{n}^{i}\right]\right|\\
&\le \sum_{i=0}^{k-1} \E\left[ |\mu(D_{h_{\ell}}^\varepsilon(t_{n}^i)){h_\ell}|\right]\\
&\le h_{\ell} \sqrt{a}\sum_{i=0}^{k-1}(1+\E[|D_{h_{\ell}}^\varepsilon(t_{n}^i)|]),
\end{align*}
where the first inequality is simply the triangle inequality, the second follows from the triangle inequality combined with the observation that the expectations of the diffusion terms are zero, and the third inequality follows from our running assumptions.
The proof is completed by using Lemma \ref{lem:moment_bound} and recalling that $k \le M$.
\end{proof}

\begin{lemma}\label{lem:local_error} 
	\[
		 \max_{\substack{0\le n\le M^{\ell-1}\\ 1\le k\le M}}  \E [|D_{h_{\ell}}^\varepsilon(t_{n}^k) - D_{h_{\ell}}^\varepsilon(t_{n})|^4|]\le C_1M^4h_{\ell}^4+C_2\varepsilon^4M^2h_{\ell}^{2}, 
	\]
where $C_1$ and $C_2$ are positive constants that only depend on $a,b,T,m,D(0)$.
 \end{lemma}
\begin{proof}
Iterating \eqref{eq:d11} yields
\begin{align*}
D_{h_{\ell}}^\varepsilon(t_{n}^{k})-D_{h_{\ell}}^\varepsilon(t_{n}) = \sum_{i=0}^{k-1}\mu(D_{h_{\ell}}^\varepsilon(t_{n}^i)){h_\ell} +  \varepsilon\sqrt{h_\ell}\sum_{i=0}^{k-1}\sigma(D_{h_{\ell}}^\varepsilon(t_{n}^i))W_{n}^{i}.
\end{align*}
Denoting  $\| X \|_{L^4(\Omega, \R^d)}= \left( \E[|X|^4]\right)^{1/4}$ and $\sigma^j$ to be the $j$th column of $\sigma$, we use the inequality $(a+b)^4 \le 8a^4  + 8b^4$ to conclude
\begin{align*}
 \E\big[ |D_{h_{\ell}}^\varepsilon(t_{n}^{k})-&D_{h_{\ell}}^\varepsilon(t_{n}) |^4\big] \le 8M^3 \sum_{i=0}^{k-1}\E\left[ |\mu(D_{h_{\ell}}^\varepsilon(t_{n}^i)){h_\ell}|^4\right] + 8 \varepsilon^4 h_{\ell}^2\E\left[ \left|\sum_{i=0}^{k-1}\sigma(D_{h_{\ell}}^\varepsilon(t_{n}^i))W_{n}^{i}\right|^4\right]\\
&\le C(a,b,T,D(0))M^4h_{\ell}^4+ 2048\varepsilon^4 h_{\ell}^2 \left(\sum_{i=0}^{M-1}\sum_{j=1}^{m}\|\sigma^j(D_{h_{\ell}}^\varepsilon(t_{n}^i))\|^2_{L^{4}(\Omega,\R^{d})}\right)^2\\
&\le C(a,b,T,D(0))^4M^4h_{\ell}^4+2048b\varepsilon^4 M^2h_{\ell}^2m^2(2+2\max_{0\le i\le M-1}\|D_{h_{\ell}}^\varepsilon(t_{n}^i)\|_{L^{4}(\Omega,\R^{d})}^4)\\
&\le C(a,b,T,D(0))^4M^4h_{\ell}^4+C_2\varepsilon^4M^2h_{\ell}^2,
\end{align*}
where the second inequality follows from Lemma \ref{lem:moment_bound} and Lemma 3.8 in \cite{hutzenthaler2012strong}, the last inequality follows from Lemma \ref{lem:moment_bound}, and  $C_1$ and $C_2$ are constants only depending on $a,b,T,m,D(0)$.
\end{proof}

The following is a Taylor expansion of the drift coefficient.
\begin{lemma}\label{lem:3terms}  Let $\mu_i(x)$ be the $i$th component of $\mu(x)$, then
	\begin{align}
	\label{eq21321214}
		 \mu_i(D_{h_{\ell}}^\varepsilon(t_{n}^k))-\mu_i(D_{h_{\ell}}^\varepsilon(t_{n}))=
A_k+B_k+E_k, 
	\end{align}
where
\begin{align*}
A_k:=\int_0^1\left[ \nabla\mu_i(D_{h_{\ell}}^\varepsilon(t_{n})+s(D_{h_{\ell}}^\varepsilon(t_{n}^k)-D_{h_{\ell}}^\varepsilon(t_{n}))) \right] ds\cdot\left(h_{\ell}\sum_{j=0}^{k-1}\mu(D_{h_{\ell}}^\varepsilon(t_{n}^j))\right),
\end{align*}
\begin{align}\label{eq:Bk}
B_k:=\nabla\mu_i(D_{h_{\ell}}^\varepsilon(t_{n}))\cdot\left(\varepsilon\sqrt{h_{\ell}}\sum_{j=0}^{k-1}\sigma(D_{h_{\ell}}^\varepsilon(t_{n}^j))W_n^j\right),
\end{align}
and
\begin{align*}
E_k:=&\left(\int_0^1\int_0^s \left[ \nabla^2\mu_i(D_{h_{\ell}}^\varepsilon(t_{n})+r(D_{h_{\ell}}^\varepsilon(t_{n}^k)-D_{h_{\ell}}^\varepsilon(t_{n})))(D_{h_{\ell}}^\varepsilon(t_{n}^k)-D_{h_{\ell}}^\varepsilon(t_{n})) \right] drds\right)\\ \notag
&\quad\cdot\left(\varepsilon\sqrt{h_{\ell}}\sum_{j=0}^{k-1}\sigma(D_{h_{\ell}}^\varepsilon(t_{n}^j))W_n^j\right).
\end{align*}
 \end{lemma}
\begin{proof}
Using  Taylor's  expansion (see Lemma \ref{lem:taylor} in the appendix) we see
\begin{align*}
    \mu_i(&D_{h_{\ell}}^\varepsilon(t_{n}^k))-\mu_i(D_{h_{\ell}}^\varepsilon(t_{n}))\\
    &=\int_0^1\left[ \nabla\mu_i(D_{h_{\ell}}^\varepsilon(t_{n})+s(D_{h_{\ell}}^\varepsilon(t_{n}^k)-D_{h_{\ell}}^\varepsilon(t_{n}))) \right]ds\cdot(D_{h_{\ell}}^\varepsilon(t_{n}^k)-D_{h_{\ell}}^\varepsilon(t_{n}))\\
    &=\int_0^1\left[\nabla\mu_i(D_{h_{\ell}}^\varepsilon(t_{n})+s(D_{h_{\ell}}^\varepsilon(t_{n}^k)-D_{h_{\ell}}^\varepsilon(t_{n}))) \right] ds\cdot\left(h_{\ell}\sum_{j=0}^{k-1}\mu(D_{h_{\ell}}^\varepsilon(t_{n}^j))\right)\\
    &\quad+\int_0^1\left[ \nabla\mu_i(D_{h_{\ell}}^\varepsilon(t_{n})+s(D_{h_{\ell}}^\varepsilon(t_{n}^k)-D_{h_{\ell}}^\varepsilon(t_{n}))) \right] ds\cdot\left(\varepsilon\sqrt{h_{\ell}}\sum_{j=0}^{k-1}\sigma(D_{h_{\ell}}^\varepsilon(t_{n}^j))W_n^j\right).
\end{align*}
Applying a multidimensional version of Lemma \ref{lem:taylor},
\begin{align*}
\int_0^1[& \nabla\mu_i(D_{h_{\ell}}^\varepsilon(t_{n})+s(D_{h_{\ell}}^\varepsilon(t_{n}^k)-D_{h_{\ell}}^\varepsilon(t_{n}))) ] ds\cdot\left(\varepsilon\sqrt{h_{\ell}}\sum_{j=0}^{k-1}\sigma(D_{h_{\ell}}^\varepsilon(t_{n}^j))W_n^j\right)\\
&=\nabla\mu_i(D_{h_{\ell}}^\varepsilon(t_{n}))\cdot\left(\varepsilon\sqrt{h_{\ell}}\sum_{j=0}^{k-1}\sigma(D_{h_{\ell}}^\varepsilon(t_{n}^j))W_n^j\right)\\
&\quad+\left(\int_0^1\int_0^s \left[ H(\mu_i) 
(D_{h_{\ell}}^\varepsilon(t_{n})+r(D_{h_{\ell}}^\varepsilon(t_{n}^k)-D_{h_{\ell}}^\varepsilon(t_{n})))(D_{h_{\ell}}^\varepsilon(t_{n}^k)-D_{h_{\ell}}^\varepsilon(t_{n})) \right] drds\right)\\
&\qquad\cdot\left(\varepsilon\sqrt{h_{\ell}}\sum_{j=0}^{k-1}\sigma(D_{h_{\ell}}^\varepsilon(t_{n}^j))W_n^j\right).
\end{align*}
Therefore,
    \begin{align*}
     \mu_i(D_{h_{\ell}}^\varepsilon(&t_{n}^k))-\mu_i(D_{h_{\ell}}^\varepsilon(t_{n}))\\ 
       &=\int_0^1[ \nabla\mu_i(D_{h_{\ell}}^\varepsilon(t_{n})+s(D_{h_{\ell}}^\varepsilon(t_{n}^k)-D_{h_{\ell}}^\varepsilon(t_{n}))) ] ds\cdot\left(h_{\ell}\sum_{j=0}^{k-1}\mu(D_{h_{\ell}}^\varepsilon(t_{n}^j))\right)\\ 
 &\quad+  \nabla\mu_i(D_{h_{\ell}}^\varepsilon(t_{n}))\cdot\left(\varepsilon\sqrt{h_{\ell}}\sum_{j=0}^{k-1}\sigma(D_{h_{\ell}}^\varepsilon(t_{n}^j))W_n^j\right)\\ 
&\quad+\left(\int_0^1\int_0^s [ \nabla^2\mu_i(D_{h_{\ell}}^\varepsilon(t_{n})+r(D_{h_{\ell}}^\varepsilon(t_{n}^k)-D_{h_{\ell}}^\varepsilon(t_{n})))(D_{h_{\ell}}^\varepsilon(t_{n}^k)-D_{h_{\ell}}^\varepsilon(t_{n})) ] drds\right)\\ 
&\qquad\cdot \left(\varepsilon\sqrt{h_{\ell}}\sum_{j=0}^{k-1}\sigma(D_{h_{\ell}}^\varepsilon(t_{n}^j))W_n^j\right)\\
&=A_k+B_k+E_k.
\qedhere
         \end{align*}
\end{proof}

The following result is similar to the $L^2$ bound  found in \cite{MT97} 
in the case where the numerical discretization method is Euler--Maruyama.

 \begin{lemma}\label{lem:L2 error} 
	\[
		 \max_{0\le n\le M^{{\ell}-1}}\E [|D_{h_{\ell}}^\varepsilon(t_{n}) - D_{h_{\ell-1}}^\varepsilon(t_{n})|^2]\le d_1 h_{\ell-1}^2+d_2\varepsilon^4 h_{\ell-1}, 
	\]
where $d_1$ and $d_2$ are positive constants that depend on $a,b,T,m,D(0)$.
 \end{lemma}
  \begin{proof}
    For $n \le M^{\ell-1}-1$ we have
    \begin{align*}
   D_{h_{\ell}}^\varepsilon(t_{n+1})-D_{h_{\ell-1}}^\varepsilon(t_{n+1}) &=D_{h_{\ell}}^\varepsilon(t_{n})-D_{h_{\ell-1}}^\varepsilon(t_{n}) + h_{\ell}\sum_{k=0}^{M-1}( \mu(D_{h_{\ell}}^\varepsilon(t_{n}^k))-\mu(D_{h_{\ell-1}}^\varepsilon(t_{n})))\\
   &\quad+\varepsilon\sqrt{h_\ell}\sum_{k=0}^{M-1} (\sigma(D_{h_{\ell}}^\varepsilon(t_{n}^k))-\sigma(D_{h_{\ell-1}}^\varepsilon(t_{n})))W_{n}^{k}\\
     &=D_{h_{\ell}}^\varepsilon(t_{n})-D_{h_{\ell-1}}^\varepsilon(t_{n}) + h_{\ell}\sum_{k=0}^{M-1}( \mu(D_{h_{\ell}}^\varepsilon(t_{n}^k))-\mu(D_{h_{\ell}}^\varepsilon(t_{n})))\\
     &\quad+h_{\ell}\sum_{k=0}^{M-1}( \mu(D_{h_{\ell}}^\varepsilon(t_{n}))-\mu(D_{h_{\ell-1}}^\varepsilon(t_{n})))\\
     &\quad +\varepsilon\sqrt{h_\ell} \sum_{k=0}^{M-1} (\sigma(D_{h_{\ell}}^\varepsilon(t_{n}^k))-\sigma(D_{h_{\ell}}^\varepsilon(t_{n})))W_{n}^{k}\\
     &\quad+\varepsilon\sqrt{h_\ell} \sum_{k=0}^{M-1} (\sigma(D_{h_{\ell}}^\varepsilon(t_{n}))-\sigma(D_{h_{\ell-1}}^\varepsilon(t_{n})))W_{n}^{k},
    \end{align*}
    where the final equality simply comes from adding and subtracting some terms.
After some manipulation the above implies
   \begin{align*}
      |D_{h_{\ell}}^\varepsilon(t_{n+1})&-D_{h_{\ell-1}}^\varepsilon(t_{n+1})|^2\le|D_{h_{\ell}}^\varepsilon(t_{n})-D_{h_{\ell-1}}^\varepsilon(t_{n}) |^2+ 4h_{\ell}^2|\sum_{k=0}^{M-1} ( \mu(D_{h_{\ell}}^\varepsilon(t_{n}^k))-\mu(D_{h_{\ell}}^\varepsilon(t_{n})))|^2\\
     &\quad+4h_{\ell}^2|\sum_{k=0}^{M-1}( \mu(D_{h_{\ell}}^\varepsilon(t_{n}))-\mu(D_{h_{\ell-1}}^\varepsilon(t_{n})))|^2 \\
     &\quad +4|\varepsilon\sqrt{h_\ell} \sum_{k=0}^{M-1} (\sigma(D_{h_{\ell}}^\varepsilon(t_{n}^k))-\sigma(D_{h_{\ell}}^\varepsilon(t_{n})))W_{n}^{k}|^2\\
     &\quad+4|\varepsilon\sqrt{h_\ell} \sum_{k=0}^{M-1} (\sigma(D_{h_{\ell}}^\varepsilon(t_{n}))-\sigma(D_{h_{\ell-1}}^\varepsilon(t_{n})))W_{n}^{k}|^2\\
     &\quad+2h_\ell\sum_{k=0}^{M-1} \langle D_{h_{\ell}}^\varepsilon(t_{n})-D_{h_{\ell-1}}^\varepsilon(t_{n}),\mu(D_{h_{\ell}}^\varepsilon(t_{n}^k))-\mu(D_{h_{\ell}}^\varepsilon(t_{n}))\rangle\\
     &\quad+ 2h_\ell\sum_{k=0}^{M-1} \langle D_{h_{\ell}}^\varepsilon(t_{n})-D_{h_{\ell-1}}^\varepsilon(t_{n}),\mu(D_{h_{\ell}}^\varepsilon(t_{n}))-\mu(D_{h_{\ell-1}}^\varepsilon(t_{n}))\rangle\\
     &\quad+2\varepsilon\sqrt{h_\ell} \sum_{k=0}^{M-1} \langle D_{h_{\ell}}^\varepsilon(t_{n})-D_{h_{\ell-1}}^\varepsilon(t_{n}) , (\sigma(D_{h_{\ell}}^\varepsilon(t^k_{n}))-\sigma(D_{h_{\ell}}^\varepsilon(t_{n})))W_{n}^{k}\rangle\\
     &\quad+2\varepsilon\sqrt{h_\ell} \sum_{k=0}^{M-1} \langle D_{h_{\ell}}^\varepsilon(t_{n})-D_{h_{\ell-1}}^\varepsilon(t_{n}) , (\sigma(D_{h_{\ell}}^\varepsilon(t_{n}))-\sigma(D_{h_{\ell-1}}^\varepsilon(t_{n})))W_{n}^{k}\rangle,
    \end{align*}
    where $\langle u, v\rangle$ denotes the inner product of $u$ and $v$.
   Therefore,
  \begin{align*}
     \E[|D_{h_{\ell}}^\varepsilon&(t_{n+1})-D_{h_{\ell-1}}^\varepsilon(t_{n+1})|^2]\\
     &\le\E[|D_{h_{\ell}}^\varepsilon(t_{n})-D_{h_{\ell-1}}^\varepsilon(t_{n}) |^2]+ 4Mh_{\ell}^2\sum_{k=0}^{M-1} \E[|\mu(D_{h_{\ell}}^\varepsilon(t_{n}^k)-\mu(D_{h_{\ell}}^\varepsilon(t_{n}))|^2]\\
     &\quad+4Mh_{\ell}^2\sum_{k=0}^{M-1} \E[|\mu(D_{h_{\ell}}^\varepsilon(t_{n}))-\mu(D_{h_{\ell-1}}^\varepsilon(t_{n}))|^2]\\
     & \quad +4\varepsilon^2h_\ell \sum_{k=0}^{M-1} \E[|(\sigma(D_{h_{\ell}}^\varepsilon(t_{n}^k))-\sigma(D_{h_{\ell}}^\varepsilon(t_{n})))W_{n}^{k}|^2]\\
     &\quad+4\varepsilon^2h_{\ell}\sum_{k=0}^{M-1} \E[|(\sigma(D_{h_{\ell}}^\varepsilon(t_{n}))-\sigma(D_{h_{\ell-1}}^\varepsilon(t_{n})))W_{n}^{k}|^2]\\
     &\quad+2h_\ell\sum_{k=0}^{M-1} \E[\langle D_{h_{\ell}}^\varepsilon(t_{n})-D_{h_{\ell-1}}^\varepsilon(t_{n}),\mu(D_{h_{\ell}}^\varepsilon(t_{n}^k))-\mu(D_{h_{\ell}}^\varepsilon(t_{n}))\rangle]\\
     &\quad+ 2h_\ell\sum_{k=0}^{M-1} \E[\langle D_{h_{\ell}}^\varepsilon(t_{n})-D_{h_{\ell-1}}^\varepsilon(t_{n}),\mu(D_{h_{\ell}}^\varepsilon(t_{n}))-\mu(D_{h_{\ell-1}}^\varepsilon(t_{n}))\rangle],
     \end{align*}
     where we used that $W_n^k$ is independent from  $D_{h_{\ell}}^\varepsilon(t_n), D_{h_{\ell-1}}^\varepsilon(t_n),$ and $D_{h_\ell}^\varepsilon(t_n^k)$.
     Hence, by Lemma \ref{lem:local_error}, there are positive constants $C_1$ and $C_2$ that only depend on $a,b,T,m,D(0)$, such that
      \begin{align*}   
      \E[|D_{h_{\ell}}^\varepsilon&(t_{n+1})-D_{h_{\ell-1}}^\varepsilon(t_{n+1})|^2]\\
     &\le \E[|D_{h_{\ell}}^\varepsilon(t_{n})-D_{h_{\ell-1}}^\varepsilon(t_{n}) |^2]+ 4aC_1M^4h_{\ell}^4+4aC_2\varepsilon^2M^3h_{\ell}^3\\
     &\quad +4aMh_{\ell}^2\sum_{k=0}^{M-1} \E[|D_{h_{\ell}}^\varepsilon(t_{n})-D_{h_{\ell-1}}^\varepsilon(t_{n}) |^2] +4bC_1\varepsilon^2M^3h_{\ell}^3+4bC_2\varepsilon^4M^2h_{\ell}^2 \\
     &\quad +4b\varepsilon^2h_\ell \sum_{k=0}^{M-1} \E[|D_{h_{\ell}}^\varepsilon(t_{n})-D_{h_{\ell-1}}^\varepsilon(t_{n}) |^2]\\
     &\quad +2h_\ell\sum_{k=0}^{M-1} \E[\langle D_{h_{\ell}}^\varepsilon(t_{n})-D_{h_{\ell-1}}^\varepsilon(t_{n}),\mu(D_{h_{\ell}}^\varepsilon(t_{n}^k))-\mu(D_{h_{\ell}}^\varepsilon(t_{n}))\rangle]\\
     &\quad + 2h_\ell \sqrt{a} M\E[|D_{h_{\ell}}^\varepsilon(t_{n})-D_{h_{\ell-1}}^\varepsilon(t_{n}) |^2],
     \end{align*}
     where the final term follows from the Cauchy-Schwarz inequality.  Continuing,
    \begin{align}
      \E[|D_{h_{\ell}}^\varepsilon&(t_{n+1})-D_{h_{\ell-1}}^\varepsilon(t_{n+1})|^2]\notag\\
     &\le \E[|D_{h_{\ell}}^\varepsilon(t_{n})-D_{h_{\ell-1}}^\varepsilon(t_{n}) |^2]+(2\sqrt{a}+4aMh_{\ell}+\varepsilon^24b)Mh_{\ell}\E[|D_{h_{\ell}}^\varepsilon(t_{n})-D_{h_{\ell-1}}^\varepsilon(t_{n}) |^2]\notag\\
     &\quad  +4aC_1M^4h_{\ell}^4+4aC_2\varepsilon^2M^3h_{\ell}^3+4bC_1\varepsilon^2M^3h_{\ell}^3+4bC_2\varepsilon^4M^2h_{\ell}^2\notag\\
      &\quad +2h_\ell\sum_{k=0}^{M-1} \E[\langle D_{h_{\ell}}^\varepsilon(t_{n})-D_{h_{\ell-1}}^\varepsilon(t_{n}),\mu(D_{h_{\ell}}^\varepsilon(t_{n}^k))-\mu(D_{h_{\ell}}^\varepsilon(t_{n}))\rangle].\label{eq:876}
    \end{align}
We turn to the term \eqref{eq:876}. Applying Lemma \ref{lem:3terms}, we know
    \begin{align*}
     &\quad\mu_i(D_{h_{\ell}}^\varepsilon(t_{n}^k))-\mu_i(D_{h_{\ell}}^\varepsilon(t_{n}))=A_k+B_k+E_k\notag.
         \end{align*}
         Also we notice,
        \[
        \E[|A_k|^2]\le K_1M^2h_{\ell}^2,
        \]
         where $K_1$ is a constant that only depends on $a,b,T,m,D(0)$. Utilizing Lemmas \ref{lem:moment_bound} and \ref{lem:local_error}
        \begin{align}
        \label{eq3214322432}
        \begin{split}
        \E[|E_k|^2]&\leq ah_{\ell}\varepsilon^2\E\left[|D_{h_{\ell}}^\varepsilon(t_{n}^k)-D_{h_{\ell}}^\varepsilon(t_{n})|^2\sum_{j=0}^{k-1}|\sigma(D_{h_{\ell}}^\varepsilon(t_{n}^j))W_n^j |^2\right]\\ 
  &\le ah_{\ell}\varepsilon^2\left(\E[|D_{h_{\ell}}^\varepsilon(t_{n}^k)-D_{h_{\ell}}^\varepsilon(t_{n})|^4]\right)^{1/2}\left(k\sum_{j=0}^{k-1}\E\left[|\sigma(D_{h_{\ell}}^\varepsilon(t_{n}^j))W_n^j |^4\right]\right)^{1/2}\\ 
  &\le K_2M^3h^3_{\ell}\varepsilon^2+K_3M^2h^2_{\ell}\varepsilon^4, 
  \end{split}
        \end{align}
  where $K_2$ and $K_3$ are constants depending only on $a,b,T,m,D(0)$. As a result,
    \begin{align}
     2h_{\ell}\sum_{k=0}^{M-1} &\E[\langle D_{h_{\ell}}^\varepsilon(t_{n})-D_{h_{\ell-1}}^\varepsilon(t_{n}),\mu(D_{h_{\ell}}^\varepsilon(t_{n}^{k}))-\mu(D_{h_{\ell}}^\varepsilon(t_{n}))\rangle]\notag \\
     &=2h_{\ell}\sum_{k=0}^{M-1} \E[\langle D_{h_{\ell}}^\varepsilon(t_{n})-D_{h_{\ell-1}}^\varepsilon(t_{n}),A_k\rangle]\notag\\
     &\quad+2h_{\ell}\sum_{k=0}^{M-1} \E[\langle D_{h_{\ell}}^\varepsilon(t_{n})-D_{h_{\ell-1}}^\varepsilon(t_{n}),B_k\rangle]\label{eq:7686}\\
       &\quad+2h_{\ell}\sum_{k=0}^{M-1} \E[\langle D_{h_{\ell}}^\varepsilon(t_{n})-D_{h_{\ell-1}}^\varepsilon(t_{n}),E_k\rangle]\notag\\
     &\le 2Mh_{\ell}\E[|D_{h_{\ell}}^\varepsilon(t_{n})-D_{h_{\ell-1}}^\varepsilon(t_{n})|^2]+h_{\ell}\sum_{k=0}^{M-1}\E [|A_k|^2]+h_{\ell}\sum_{k=0}^{M-1}\E [|E_k|^2]\notag\\
     &\le 2Mh_{\ell}\E[|D_{h_{\ell}}^\varepsilon(t_{n})-D_{h_{\ell-1}}^\varepsilon(t_{n})|^2]+ K_1M^3h_{\ell}^3+K_2M^4h^4_{\ell}\varepsilon^2+K_3M^3h^3_{\ell}\varepsilon^4,\notag
    \end{align}
where the first inequality follows from: (i) the observation that the expectation \eqref{eq:7686}  is zero, (ii) the Cauchy-Schwarz inequality, and (iii) the inequality $2ab\le a^2+b^2$.   Combining all the estimates above, we find
   \begin{align*}
    \E[|D_{h_{\ell}}^\varepsilon&(t_{n+1})-D_{h_{\ell-1}}^\varepsilon(t_{n+1})|^2]\\
     &\le\E\left[|D_{h_{\ell}}^\varepsilon(t_{n})-D_{h_{\ell-1}}^\varepsilon(t_{n}) |^2\right]\\
     &\quad+(2+2\sqrt{a}+4aMh_{\ell}+4b\varepsilon^2)Mh_{\ell}\E\left[|D_{h_{\ell}}^\varepsilon(t_{n})-D_{h_{\ell-1}}^\varepsilon(t_{n}) |^2\right]\\
     &\quad +4aC_1M^4h_{\ell}^4+4aC_2\varepsilon^2M^3h_{\ell}^3+4bC_1\varepsilon^2M^3h_{\ell}^3+4bC_2\varepsilon^4M^2h_{\ell}^2\\
      &\quad +K_1M^3h_{\ell}^3+K_2M^4h^4_{\ell}\varepsilon^2+K_3M^3h^3_{\ell}\varepsilon^4.
   \end{align*}
   Noting that the dominant terms above are of order $h_{\ell-1}^2 \varepsilon^4$ and $h_{\ell-1}^3$, an application of Gronwall's inequality completes the proof.
  \end{proof}

	We are now ready to prove our main result.

\begin{proof}[Proof of Theorem \ref{thm:var}]
		Following \cite{AHS2014}, we first prove the result in the case that $f(x) = x_i$ for some $i \in \{1,\dots,d\}$.  We have that for $n \le M^{\ell-1}-1$,
    \begin{align*}
   [&D_{h_{\ell}}^\varepsilon(t_{n+1})-D_{h_{\ell-1}}^\varepsilon(t_{n+1})]_i\\
     &=[D_{h_{\ell}}^\varepsilon(t_{n})-D_{h_{\ell-1}}^\varepsilon(t_{n})]_i + h_{\ell}\sum_{k=0}^{M-1} ( \mu_i(D_{h_{\ell}}^\varepsilon(t_{n}^k))-\mu_i(D_{h_{\ell}}^\varepsilon(t_{n})))\\
     &\quad+h_{\ell}\sum_{k=0}^{M-1} ( \mu_i(D_{h_{\ell}}^\varepsilon(t_{n}))-\mu_i(D_{h_{\ell-1}}^\varepsilon(t_{n}))) +\varepsilon\sqrt{h_\ell} \sum_{k=0}^{M-1} (\sigma_i(D_{h_{\ell}}^\varepsilon(t_{n}^k))-\sigma_i(D_{h_{\ell}}^\varepsilon(t_{n})))W_{n}^{k}\\
     &\quad+\varepsilon\sqrt{h_\ell} \sum_{k=0}^{M-1} (\sigma_iD_{h_{\ell}}^\varepsilon(t_{n}))-\sigma_i(D_{h_{\ell-1}}^\varepsilon(t_{n})))W_{n}^{k},
    \end{align*}
where $\mu_i$ is the $i$th component of $\mu$ and $\sigma_i$ is the $i$th row of $\sigma$.
As a result, and after some manipulation,
\begin{align}
  \textsf{Var}&([D_{h_{\ell}}^\varepsilon(t_{n+1})-D_{h_{\ell-1}}^\varepsilon(t_{n+1})]_i)\notag\\ 
  &\le (1+Mh_\ell) \textsf{Var}([D_{h_{\ell}}^\varepsilon(t_{n})-D_{h_{\ell-1}}^\varepsilon(t_{n})]_i) \notag\\
  &\quad +4 h_{\ell}^2M\sum_{k=0}^{M-1} \textsf{Var}( \mu_i(D_{h_{\ell}}^\varepsilon(t_{n}^k))-\mu_i(D_{h_{\ell}}^\varepsilon(t_{n}))) \label{eq:AA1}\\ 
     &\quad +(4Mh_{\ell}+1)Mh_{\ell}\textsf{Var}( \mu_i(D_{h_{\ell}}^\varepsilon(t_{n}))-\mu_i(D_{h_{\ell-1}}^\varepsilon(t_{n})))\label{eq:AA2}\\ 
      &\quad  +4\varepsilon^2h_{\ell} \sum_{k=0}^{M-1} \textsf{Var}((\sigma_i(D_{h_{\ell}}^\varepsilon(t_{n}^k))-\sigma_i(D_{h_{\ell}}^\varepsilon(t_{n})))W_{n}^{k})\label{eq:AA3}\\ 
     &\quad +4\varepsilon^2h_{\ell} \sum_{k=0}^{M-1} \textsf{Var}( (\sigma_i(D_{h_{\ell}}^\varepsilon(t_{n}))-\sigma_i(D_{h_{\ell-1}}^\varepsilon(t_{n})))W_{n}^{k})\label{eq:AA4}\\ 
     &\quad +2\textsf{Cov}([D_{h_{\ell}}^\varepsilon(t_{n})-D_{h_{\ell-1}}^\varepsilon(t_{n})]_i,h_{\ell}\sum_{k=0}^{M-1} ( \mu_i(D_{h_{\ell}}^\varepsilon(t_{n}^k))-\mu_i(D_{h_{\ell}}^\varepsilon(t_{n}))))\label{eq:AA5}. 
    \end{align}
	 We must bound each expression on the right-hand side in order to apply Gronwall's inequality.  We first consider \eqref{eq:AA3}, which leads to a dominant term.  Lemma~\ref{lem:local_error} implies that
\begin{align*}
	\sum_{k=0}^{M-1} \textsf{Var}((\sigma_i(D_{h_{\ell}}^\varepsilon(t_{n}^k))-\sigma_i(D_{h_{\ell}}^\varepsilon(t_{n})))W_{n}^{k})  & \le\sum_{k=0}^{M-1} \E[|\sigma_i(D_{h_{\ell}}^\varepsilon(t_{n}^k))-\sigma_i(D_{h_{\ell}}^\varepsilon(t_{n})))|^2]\\
           &\le Mb(c_1M^2h_{\ell}^2+c_2\varepsilon^2Mh_{\ell}).
\end{align*}
Similarly, by Lemma~\ref{lem:L2 error} we  may bound \eqref{eq:AA4}, which also yields a dominant term,
		 \begin{align*}
 \sum_{k=0}^{M-1} \textsf{Var}( (\sigma_i(D_{h_{\ell}}^\varepsilon(t_{n}))-\sigma_i(D_{h_{\ell-1}}^\varepsilon(t_{n})))W_{n}^{k}) &\le \sum_{k=0}^{M-1} \E[|\sigma_i(D_{h_{\ell}}^\varepsilon(t_{n}))-\sigma_i(D_{h_{\ell}-1}^\varepsilon(t_{n})))|^2]\\
           &\le Mb(d_1M^2h_{\ell}^2+d_2\varepsilon^4Mh_{\ell}),
		\end{align*}
where $c_1,c_2,d_1$ and $d_2$ are positive constants only depending on $a,b,T,m,D(0)$.

\vspace{.1in}

		Turning to \eqref{eq:AA1}, we have the following lemma.
		
\begin{lemma}\label{lemma:third_term}
	\begin{align*}&\displaystyle \textsf{Var}\left(  \mu_i(D_{h_{\ell}}^\varepsilon(t_{n}^k))-\mu_i(D_{h_{\ell}}^\varepsilon(t_{n}))\right)\le CMh_{\ell}\varepsilon^2,
            \end{align*}
 where $C$ is a positive constant that only depends on $a,b,T,d,m,D(0)$.
 \end{lemma}
		
		\begin{proof}
      From Lemma \ref{lem:taylor} in the appendix (Taylor approximation), we have
       \begin{align}
       \mu_i(D_{h_{\ell}}^\varepsilon(t_{n}^k))-\mu_i&(D_{h_{\ell}}^\varepsilon(t_{n})) = \rho^k(t_n) \cdot (D_{h_{\ell}}^\varepsilon(t_{n}^k)-D_{h_{\ell}}^\varepsilon(t_{n})),
              \label{eq:crazy23423}
       \end{align}
       where
       \begin{align*}
       \rho^k(t_n)= \int_0^1 \left[ \nabla \mu_i(D_{h_{\ell}}^\varepsilon(t_{n})) + r(D_{h_{\ell}}^\varepsilon(t_{n}^k)-D_{h_{\ell}}^\varepsilon(t_{n})) \right] dr.
       \end{align*}
          In order to bound the right hand side of \eqref{eq:crazy23423},  we will apply Lemma \ref{lem:Var_bound} in the appendix with $A^{\varepsilon,h_{\ell-1}}=[\rho^k(t_n)]_j$  and $B^{\varepsilon,h_{\ell-1}}=[D_{h_{\ell}}^\varepsilon(t_{n}^k)-D_{h_{\ell}}^\varepsilon(t_{n})]_j$.  Hence, we must find appropriate bounds on these components. 
          
          We begin with   $B^{\varepsilon, h_{\ell-1}}$.  We use Lemmas  \ref{lem:moment_bound} and \ref{lem:var_bound} after iterating \eqref{eq:d11} to find
     \begin{align}
     \begin{split}
           \textsf{Var}&([D_{h_{\ell}}^\varepsilon(t_{n}^k)-D_{h_{\ell}}^\varepsilon(t_{n})]_j)\\
           &\le2\textsf{Var}\left(h_{\ell}\sum_{r=0}^{k-1}\mu_j(D_{h_{\ell}}^\varepsilon(t_{n}^r))\right)+2\textsf{Var}\left(\varepsilon\sqrt{h_{\ell}}\sum_{r=0}^{k-1} \sigma_{j}(D_{h_{\ell}}^\varepsilon(t_{n}^r))W_n^r\right)\\
           &\le2h_{\ell}^2\textsf{Var}\left(\sum_{r=0}^{k-1}(\mu_j(D_{h_{\ell}}^\varepsilon(t_{n}^r))-\mu_j(z_{h_{\ell}}(t_{n}^r)))\right)+2\varepsilon^2h_{\ell}\E\left[ \left|\sum_{r=0}^{k-1} \sigma_{j}(D_{h_{\ell}}^\varepsilon(t_{n}^r))W_n^r\right|^2\right]\\
               &\le C_1M^2h_{\ell}^2\varepsilon^2+C_2Mh_{\ell}\varepsilon^2, 
           \end{split}
           \label{eq:908789}
           \end{align}
 where $C_1$ and $C_2$ are positive constants that only depend on $a,b,T,m,D(0)$.

        Turning to $A^{\varepsilon,h_{\ell-1}}$, we  apply Lemma \ref{lem:orderv} in the appendix with 
  $X_1(s)=D_{h_{\ell}}^\varepsilon(s)$, $X_2(s)=D_{h_{\ell}}^\varepsilon(\eta_{h_{\ell}}(s))$, $x_1(s)=z_{h_\ell}(s)$, $x_2(s)=z_{h_\ell}(\eta_{h_{\ell}}(s))$ and $u(x)=\nabla_j\mu_i(x)$
          to obtain
           \begin{align*}
           \textsf{Var}([\rho^k(t_n)]_j) =\textsf{Var} &\left(\int_0^1 [ \nabla_j \mu_i(D_{h_{\ell}}^\varepsilon(t_{n}) + r(D_{h_{\ell}}^\varepsilon(t_{n}^k)-D_{h_{\ell}}^\varepsilon(t_{n}))) ]dr \right)\le  K\varepsilon^2, 
        \end{align*}
where $K$ is positive constant that only depends on $a,b,T,m,D(0)$.
     
     We may now combine  Lemma \ref{lem:Var_bound} with Lemma \ref{lem:local_error1} to conclude
      \begin{align*}
   \textsf{Var}\left( [\rho^k(t_n)]_j \, \cdot\, [D_{h_{\ell}}^\varepsilon(t_{n}^k)-D_{h_{\ell}}^\varepsilon(t_{n})]_j\right) &\leq \hat{C}KM^2h_{\ell}^2\varepsilon^2  +15a\textsf{Var}([D_{h_{\ell}}^\varepsilon(t_{n}^k)-D_{h_{\ell}}^\varepsilon(t_{n})]_j)\\
   &\hspace{.3in}\leq  (\hat{C}K+15aC_1)M^2h_{\ell}^2\varepsilon^2+15C_2Mh_{\ell}\varepsilon^2\\
   &\hspace{.3in}\leq  \hat{C}_1Mh_{\ell}\varepsilon^2, 
   \end{align*}
  where $\hat{C}_1$ is positive and does not depend on $\varepsilon$ and $h_{\ell}$, and we applied \eqref{eq:908789} in the second inequality.

   Returning to \eqref{eq:crazy23423}, the above allows us to conclude
   \begin{equation*}
   \textsf{Var}(\mu_i(D_{h_{\ell}}^\varepsilon(t_{n}^k))-\mu_i(D_{h_{\ell}}^\varepsilon(t_{n})))\le d^2\hat{C}_1Mh_{\ell}\varepsilon^2. \qedhere
   \end{equation*}
   \end{proof}

		We now turn to the first term of \eqref{eq:AA2}.
\begin{lemma}\label{lemma:4_term}
	\begin{align*}
	\textsf{Var}&\left(  \mu_i(D_{h_{\ell}}^\varepsilon(t_{n}))-\mu_i(D_{h_{\ell}-1}^\varepsilon(t_{n}))\right)\\
	&\le 15ad\sum_{j=1}^d \textsf{Var}([D_{h_{\ell}}^\varepsilon(t_{n})-D_{h_{\ell-1}}^\varepsilon(t_{n})]_j)+K_1 M^2h^2_{\ell}\varepsilon^2+K_2 Mh_{\ell}\varepsilon^6,
            \end{align*}
         where $K_1,K_2$ are positive constants that only depend on $a,b,T,d,m,D(0)$.
 \end{lemma} 
 \begin{proof}
 We first  write
		\begin{align*}
			\mu_i(D_{h_{\ell}}^\varepsilon(t_{n}))-\mu_i(D_{h_{\ell}-1}^\varepsilon(t_{n}))= \rho(t_n) \cdot (D_{h_{\ell}}^\varepsilon(t_n) - D_{h_{\ell}-1}^\varepsilon(t_{n})),
		\end{align*}
		where
		\[
			\rho(t_n)= \int_0^1 [  \nabla \mu_i( D_{h_{\ell}-1}^\varepsilon(t_{n}) + r( D_{h_{\ell}}^\varepsilon(t_n) - D_{h_{\ell}-1}^\varepsilon(t_{n}))) ]dr.
		\]
		We will again apply Lemma \ref{lem:Var_bound} to get the necessary bounds.  Therefore, we let
		 $A^{\varepsilon,h}=[\rho(t_n)]_j$ and $B^{\varepsilon,h}=[D_{h_{\ell}}^\varepsilon(t_n) - D_{h_{\ell}-1}^\varepsilon(t_{n})]_j$.
				
        Letting $X_1(s)=D_{h_{\ell}}^\varepsilon(s)$, $X_2(s)=D_{h_{\ell-1}}^\varepsilon(s)$, $x_1(s)=z_{h_{\ell}}(s)$, $x_2(s)=z_{h_{\ell-1}}(s)$ and $u(x)=\nabla_j \mu_i(x)$ for an application of Lemma \ref{lem:orderv},  we have
       \begin{align}
        \textsf{Var}(A^{\varepsilon,h})   &\leq  K\varepsilon^2\notag,
        \end{align}
  for some $K(a,b,T,m,D(0))$, where we recall the running assumption that $|[\nabla \mu_i]_j|^2$ is uniformly bounded by $a$.
 Hence, applying Lemmas \ref{lem:L2 error}  and \ref{lem:Var_bound} we see there are positive constants $K_1,K_2$ depending only on $a, b, T, m, D(0)$, such that,
        \begin{align*}
          \textsf{Var}&([\rho(t_n)]_j([D_{h_{\ell}}^\varepsilon(t_n) - D_{h_{\ell}-1}^\varepsilon(t_{n})]_j)) \\
          &\leq K_1 M^2h^2_{\ell}\varepsilon^2+K_2 Mh_{\ell}\varepsilon^6+15a\textsf{Var}([D_{h_{\ell}}^\varepsilon(t_n) - D_{h_{\ell}-1}^\varepsilon(t_{n})]_j,
        \end{align*}
        and
\begin{align*}
			\textsf{Var}&(\mu_i(D_{h_{\ell}}^\varepsilon(t_n) ) - \mu_i(D_{h_{\ell-1}}^\varepsilon(t_n) )) \\
&\le 15ad\sum_{j=1}^d \textsf{Var}([D_{h_{\ell}}^\varepsilon(t_n)-D_{h_{\ell-1}}^\varepsilon(t_n)]_j)+d^2K_1 M^2h^2_{\ell}\varepsilon^2+d^2K_2 Mh_{\ell}\varepsilon^6.\notag
		\end{align*}
\end{proof}

Finally, we turn to the term \eqref{eq:AA5}.
\begin{lemma}
	\begin{align*}
		\textsf{Cov}\bigg([D_{h_{\ell}}^\varepsilon&(t_{n})-D_{h_{\ell-1}}^\varepsilon(t_{n})]_i,h_{\ell}\sum_{k=0}^{M-1} ( \mu_i(D_{h_{\ell}}^\varepsilon(t_{n}^k))-\mu_i(D_{h_{\ell}}^\varepsilon(t_{n}))) \bigg)\\
	&\le Mh_{\ell}\textsf{Var}([D_{h_{\ell}}^\varepsilon(t_{n})-D_{h_{\ell-1}}^\varepsilon(t_{n})]_i)+K_1M^3h^3_{\ell}\varepsilon^2+K_2M^5h^5_{\ell}\varepsilon^2+K_3M^3h^3_{\ell}\varepsilon^4,
            \end{align*}
 where $K_1,K_2$ and $K_3$ are positive constants that only depend on $a,b,T,m,D(0)$.
 \end{lemma} 
 \begin{proof}
 As a result of combining \eqref{eq21321214} in Lemma \ref{lem:3terms}
  with 
 \begin{align*}
 	\textsf{Cov}\left([D_{h_{\ell}}^\varepsilon(t_{n})-D_{h_{\ell-1}}^\varepsilon(t_{n})]_i,h_{\ell}\sum_{k=0}^{M-1} B_k\right)=0,
         \end{align*}
where we recall the definition of $B_k$ in \eqref{eq:Bk}, we have
\begin{align}\label{eqcov}
\begin{split}
\textsf{Cov}\bigg([D_{h_{\ell}}^\varepsilon&(t_{n})-D_{h_{\ell-1}}^\varepsilon(t_{n})]_i,h_{\ell}\sum_{k=0}^{M-1} ( \mu_i(D_{h_{\ell}}^\varepsilon(t_{n}^k))-\mu_i(D_{h_{\ell}}^\varepsilon(t_{n})))\bigg)\\ 
&= \textsf{Cov}\left([D_{h_{\ell}}^\varepsilon(t_{n})-D_{h_{\ell-1}}^\varepsilon(t_{n})]_i,h_{\ell}\sum_{k=0}^{M-1} (A_k+E_k)\right)\\
&\hspace{.2in}+\textsf{Cov}\left([D_{h_{\ell}}^\varepsilon(t_{n})-D_{h_{\ell-1}}^\varepsilon(t_{n})]_i,h_{\ell}\sum_{k=0}^{M-1} B_k\right)\\ 
&\le Mh_{\ell}\textsf{Var}([D_{h_{\ell}}^\varepsilon(t_{n})-D_{h_{\ell-1}}^\varepsilon(t_{n})]_i)+\frac{1}{2}h_{\ell}\sum_{k=0}^{M-1}\textsf{Var}\left(A_k\right)+\frac{1}{2}h_{\ell}\sum_{k=0}^{M-1}\textsf{Var}\left(E_k\right).
\end{split}
            \end{align}
 First we want to estimate $\textsf{Var}\left(A_k\right)$. Applying Lemma \ref{lem:Var_bound} with
        \begin{equation*}
        	A^{\varepsilon,h_{\ell-1}}=\int_0^1 \left[ \nabla _j\mu_i(D_{h_{\ell}}^\varepsilon(t_{n})+ r(D_{h_{\ell}}^\varepsilon(t_{n}^k)-D_{h_{\ell}}^\varepsilon(t_{n}))) \right] dr
	\end{equation*}
   and \begin{equation*}
   	B^{\varepsilon,h_{\ell-1}} =h_{\ell}\sum_{r=0}^{k-1}\mu_j(D_{h_{\ell}}^\varepsilon(t_{n}^r)),
	\end{equation*}
   we can get for some $K_1(a,b,T,m,d,D(0))$ that may change from line to line,
  \begin{align*}
  \quad\textsf{Var}(A_k)&\leq K_1M^2h_{\ell}^2\varepsilon^2  +15ad\sum_{i=1}^{d}\textsf{Var}\left(h_{\ell}\sum_{r=0}^{k-1}\mu_i(D_{h_{\ell}}^\varepsilon(t_{n}^r))\right)\\\
   &\leq  K_1M^2h_{\ell}^2\varepsilon^2+15ad\sum_{i=1}^{d}\E\left[h_{\ell}^2\left(\sum_{r=0}^{k-1}\mu_i(D_{h_{\ell}}^\varepsilon(t_{n}^r))-\mu_j(z_{h_{\ell}}(t_{n}^r))\right)^2\right]\\
   &\leq  K_1M^2h^2_{\ell}\varepsilon^2, 
  \end{align*} 
  where we also use Lemma \ref{lem:var_bound} for the last line.
On the other hand, from \eqref{eq3214322432}
\begin{align*}
  \quad\textsf{Var}(E_k)&\leq \E[|E_k|^2]\le K_2M^4h^4_{\ell}\varepsilon^2+K_3M^2h^2_{\ell}\varepsilon^4.
  \end{align*} 
  Returning to \eqref{eqcov}, we see,
  \begin{align*}
  \textsf{Cov}&\left([D_{h_{\ell}}^\varepsilon(t_{n})-D_{h_{\ell-1}}^\varepsilon(t_{n})]_i,h_{\ell}\sum_{k=0}^{M-1} ( \mu_i(D_{h_{\ell}}^\varepsilon(t_{n}^k))-\mu_i(D_{h_{\ell}}^\varepsilon(t_{n})))\right)\\ \notag
&\le Mh_{\ell}\sum_{j=1}^d\textsf{Var}([D_{h_{\ell}}^\varepsilon(t_{n})-D_{h_{\ell-1}}^\varepsilon(t_{n})]_j)+\frac{1}{2}K_1M^3h^3_{\ell}\varepsilon^2+\frac{1}{2}K_2M^5h^5_{\ell}\varepsilon^2+\frac{1}{2}K_3M^3h^3_{\ell}\varepsilon^4.
            \end{align*}
      \end{proof}	
      
Now we return to \eqref{eq:AA1}--\eqref{eq:AA5} and combine all the estimates above to conclude that there exist $C_1,C_2,$ and $C_3$ which only depend on $a,b,T,m,d,D(0)$ such that
 \begin{align*}
   \textsf{Var}\left( [D_{h_{\ell}}^\varepsilon(t_{n+1})-D_{h_{\ell-1}}^\varepsilon(t_{n+1})]_i \right)   &\le\textsf{Var}[D_{h_{\ell}}^\varepsilon(t_{n})-D_{h_{\ell-1}}^\varepsilon(t_{n})]_i +C_1M^3h_{\ell}^3\varepsilon^2+C_2M^2h_{\ell}^2\varepsilon^4\\
 &\quad   +C_3Mh_{\ell}\sum_{j=1}^d\textsf{Var}([D_{h_{\ell}}^\varepsilon(t_{n})-D_{h_{\ell-1}}^\varepsilon(t_{n})]_j).
     \end{align*}
Therefore,
   \begin{align*}
 &  \max_{i=1,2,\cdots,d}\textsf{Var}\left( [D_{h_{\ell}}^\varepsilon(t_{n+1})-D_{h_{\ell-1}}^\varepsilon(t_{n+1})]_i \right)\\ 
     &\hspace{.5in} \le\max_{i=1,2,\cdots,d}\textsf{Var}\left( [D_{h_{\ell}}^\varepsilon(t_{n})-D_{h_{\ell-1}}^\varepsilon(t_{n})]_i\right) +C_1M^3h_{\ell}^3\varepsilon^2+C_2M^2h_{\ell}^2\varepsilon^4\\
  &\hspace{.8in}  + C_3dMh_{\ell}\max_{i=1,2,\cdots,d}\textsf{Var}\left([D_{h_{\ell}}^\varepsilon(t_{n})-D_{h_{\ell-1}}^\varepsilon(t_{n})]_i\right).
     \end{align*}
     Applying Gronwall's lemma, we obtain,
    \begin{align*}
   &\quad\max_{0\le n \le M^{\ell-1}} \max_{1\le i \le d}\textsf{Var}\left([D_{h_{\ell}}^\varepsilon(t_{n})-D_{h_{\ell-1}}^\varepsilon(t_{n})]_i\right)\le C_1M^2h^2_{\ell}\varepsilon^2+C_2Mh_{\ell}\varepsilon^4, 
     \end{align*}  
where $C_1$ and $C_2$ are some universal constants which only depend on $a,b,T,m,d,D(0)$.

		\vspace{.1in}
		
		We have shown the result under the assumption that $f(x) = x_i$.  
		To show the general case, note that from Lemma \ref{lem:taylor} in the appendix we have
		\begin{align*}
			f(D_{h_{\ell}}^\varepsilon(t_n)) -& f(D_{h_{\ell-1}}^\varepsilon(t_n))) \\
			&= \int_0^1 [\nabla f(D_{h_{\ell}}^\varepsilon(t_n) + r(D_{h_{\ell}}^\varepsilon(t_n)-D_{h_{\ell}}^\varepsilon(t_n)))]dr\cdot (D_{h_{\ell}}^\varepsilon(t_n)-D_{h_{\ell-1}}^\varepsilon(t_n))\notag.
		\end{align*}
          We let $X_1(t)=D_{h_{\ell}}^\varepsilon(t) $, $X_2(t)=D_{h_{\ell-1}}^\varepsilon(t) $, $x_1(t)=z_{h_{\ell}}(t)$, $x_2(t)=z_{h_{\ell-1}}(t)$ and $u(x)= \nabla_j f(x)$ in an application of Lemma \ref{lem:orderv} which yields
          \begin{align*}
        &\textsf{Var}\left(\int_0^1 [ \nabla_jf(D_{h_{\ell}}^\varepsilon(t_n) + r(D_{h_{\ell}}^\varepsilon(t_n)-D_{h_{\ell}}^\varepsilon(t_n))) ] dr\right) \leq  K\varepsilon^2,
        \end{align*}
        where $K$ is a universal constant that depends on $C_L,D,a,b,T,D(0)$.
 Hence, by an application of Lemmas \ref{lem:L2 error} and \ref{lem:Var_bound} and the work above we see,  
        \begin{align*}
          \textsf{Var}&\left(\int_0^1 \nabla_j f(D_{h_{\ell}}^\varepsilon(t_n) + r(D_{h_{\ell}}^\varepsilon(t_n)-D_{h_{\ell}}^\varepsilon(t_n)))dr\cdot [D_{h_{\ell}}^\varepsilon(t_n)-D_{h_{\ell-1}}^\varepsilon(t_n)]_j\right)\\
          &\leq  K(d_1M^2 h_{\ell}^2+d_2M h_{\ell}\varepsilon^4)\varepsilon^2+15dC_L^2\textsf{Var}([D^\varepsilon(s)-D_{h_{\ell}}^\varepsilon(s)]_j )\\
          &\leq  (Kd_1+15dC_L^2C_1)M^2h_{\ell}^2\varepsilon^2+(15dC_L^2C_2+Kd_2)Mh_{\ell}\varepsilon^4.
        \end{align*}
  Thus
        \[
          \textsf{Var}(f(D_{h_{\ell}}^\varepsilon(t_n)) - f(D_{h_{\ell-1}}^\varepsilon(t_n))) \leq d^2(Kd_1+15dC_L^2C_1)M^2h_{\ell}^2\varepsilon^2+d^2(15dC_L^2C_2+Kd_2)h_{\ell}\varepsilon^4,
        \]
	giving the result.
	\end{proof}

\section{Numerical examples and comparison with results related to jump processes}

In this section we provide numerical evidence for the sharpness of both Theorem \ref{thm:var} and the computational complexity analyses provided in sections \ref{sec:2.1} and \ref{subsec:emmc}.  Further, we compare our results to those found in \cite{AHS2014} for scaled Markov processes.

\begin{example}\label{example1}
 We consider the following simple one dimensional model,
\begin{align*}
D^\varepsilon(t)=1-\int_{0}^{t}D^\varepsilon(s)ds+\varepsilon \int_{0}^{t}D^\varepsilon(s)dW(s),
\end{align*}
where we simulate until $T =1$.  

To gather evidence in support of the sharpness of the bound 
 $\textsf{Var}(D_{h_{\ell}}^\varepsilon(t) - D_{h_{\ell-1}}^\varepsilon(t)) = O(h^2 \varepsilon^2 + h \varepsilon^4)$, we fix one of $h$ or $\varepsilon$ in different scaling regimes and vary the other parameter in order to generate log-log plots.    We note that there are four exponents to discover, and so four log-log plots are used.  Note also that $h^2\varepsilon^2$ is the dominant term in $h^2 \varepsilon^2 + h \varepsilon^4$ if and only if $h \ge \varepsilon^2$.
We emphasize that these experiments use extreme parameter choices solely for the purpose
of testing the sharpness of the delicate asymptotic bounds. 

\vspace{.1in}

\noindent \textbf{The exponent of $h$ in $h\varepsilon^4$.}
We fix $\varepsilon = 2^{-6}$ and vary     
\[
	h_{\ell-1}\in\{2^{-13},2^{-14},2^{-15},2^{-16},2^{-17},2^{-18}\}
\]
 to ensure $h_{\ell-1}\le \varepsilon^2$.  As a result, $h_{\ell-1}\varepsilon^4$ is likely to be the dominant  term in \eqref{eq:main}. See Figure \ref{fig:example_simple1}(a), where the log-log  plot is consistent with the functional form
\begin{align*}
			\textsf{Var}(D_{h_{\ell}}^\varepsilon(T) - D_{h_{\ell-1}}^\varepsilon(T)) = O(h_{\ell-1}).
\end{align*}

\begin{figure}
\centering
\subfigure[$\varepsilon = 2^{-6}$ fixed while $h$ is varied.  The best fit curve is $y=1.02x-19.76$.]{\includegraphics[width=0.47\textwidth]{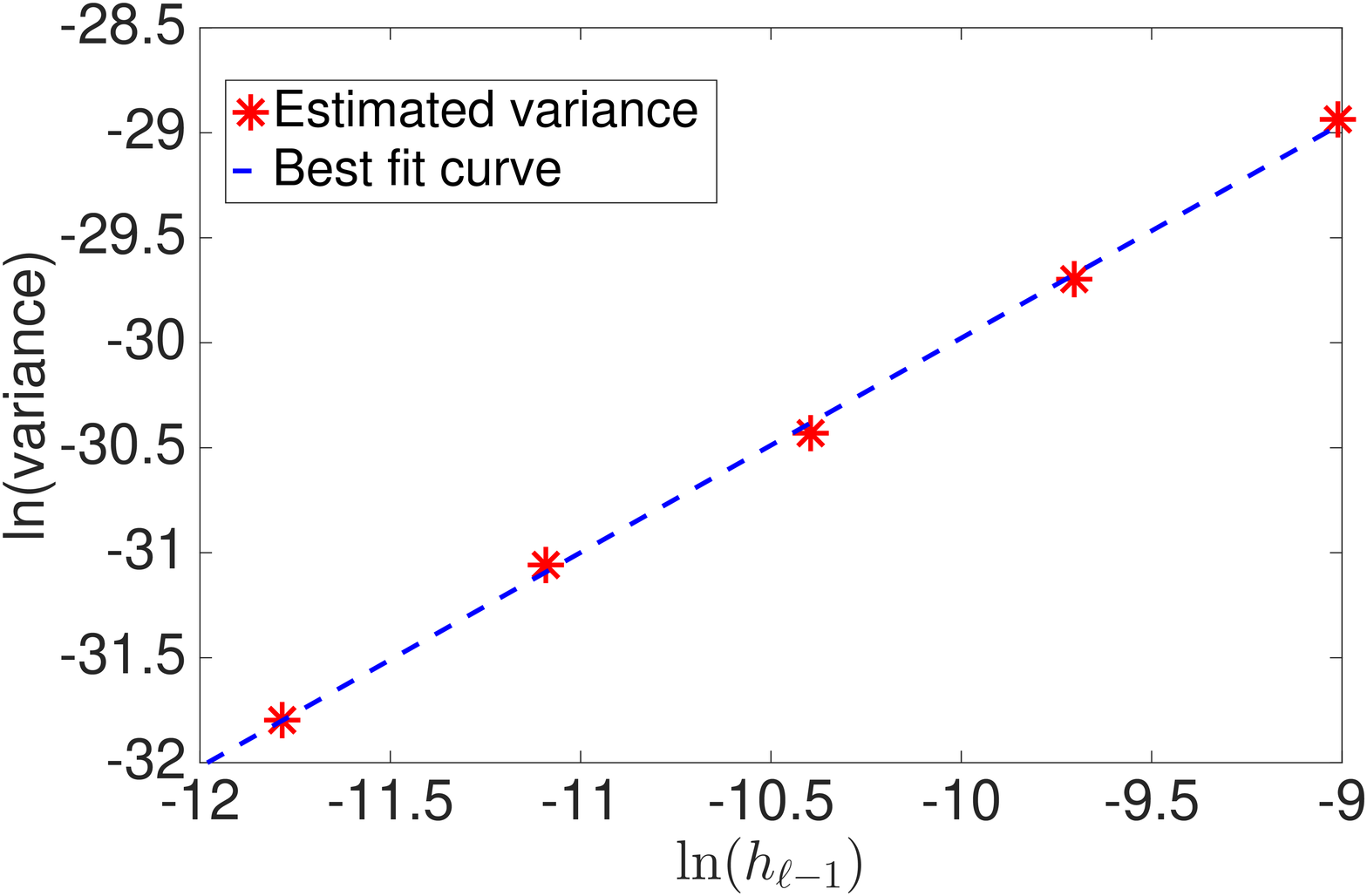}}
\hspace{.2in}
\subfigure[$\varepsilon = 2^{-10}$ fixed while $h$ is varied.  The best fit curve is $y=1.96x-18.86$.]{\includegraphics[width=0.47\textwidth]{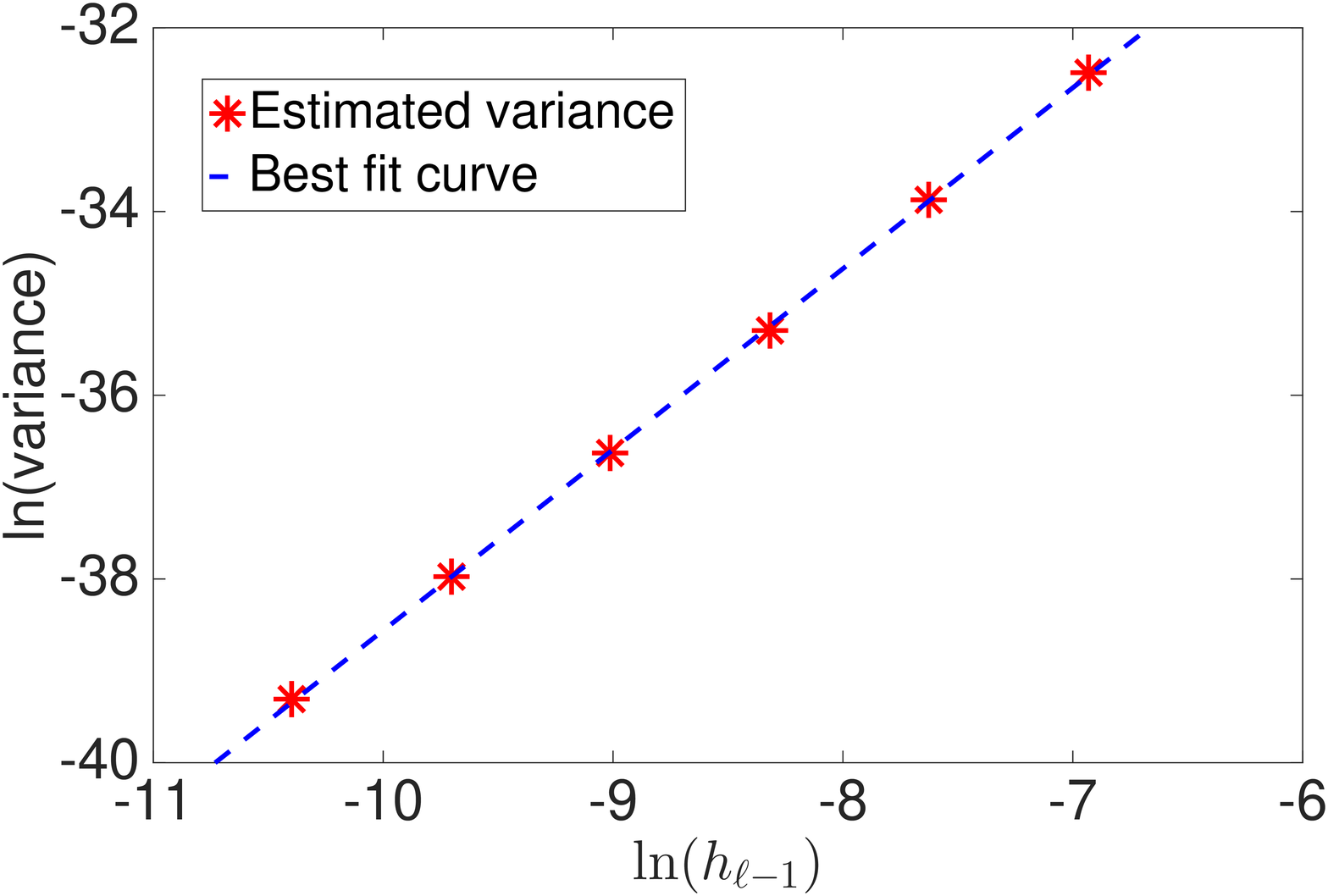}}
	\caption{Log-log plots of $\textsf{Var}(D_{h_{\ell}}^\varepsilon(1) - D_{h_{\ell-1}}^\varepsilon(1))$ with $\varepsilon$  held constant and $h_{\ell-1}$ varied. The best fit curves for all data are overlain in the dashed blue line. Each data point in (a) was generated using 2,000 independent samples and each data point in (b) was generated using 5,000 independent samples.}
	\label{fig:example_simple1}
\end{figure}

\vspace{.1in}

\noindent \textbf{The exponent of $h$ in $h^2\varepsilon^2$.} 
We fix $\varepsilon = 2^{-10}$ and vary   
\[
	h_{\ell-1}\in\{2^{-10},2^{-11},2^{-12},2^{-13},2^{-14},2^{-16}\}
	\]
	 to ensure $h_{\ell-1}\ge \varepsilon^2$.  As a result, $h_{\ell-1}^2 \varepsilon^2$ is likely to be the dominant  term in \eqref{eq:main}. See Figure \ref{fig:example_simple1}(b), where the log-log  plot is consistent with the functional form
\begin{align*}
			\textsf{Var}(D_{h_{\ell}}^\varepsilon(1) - D_{h_{\ell-1}}^\varepsilon(1)) = O(h_{\ell-1}^2).
\end{align*}

\vspace{.1in}

\noindent \textbf{The exponent of $\varepsilon$ in $h\varepsilon^4$.} 
We fix $h_{\ell-1}= 2^{-19}$  and vary   
\[
	\varepsilon \in \{2^{-5},2^{-6},2^{-7},2^{-8},2^{-9}\}
	\]
	 to ensure $h_{\ell-1}\le \varepsilon^2$.  As a result, $h_{\ell-1} \varepsilon^4$ is
likely to be  the dominant  term in \eqref{eq:main}. See Figure \ref{fig:example_simple2}(a), where the log-log  plot is consistent with the functional form
\begin{align*}
			\textsf{Var}(D_{h_{\ell}}^\varepsilon(1) - D_{h_{\ell-1}}^\varepsilon(1)) = O(\varepsilon^4).
\end{align*}

\vspace{.1in}

\noindent \textbf{The exponent of $\varepsilon$ in $h^2\varepsilon^2$.} 
We fix $h_{\ell-1}= 2^{-9}$  and vary   
\[
	\varepsilon \in \{2^{-6},2^{-7},2^{-8},2^{-9},2^{-10},2^{-11}\}
	\]
	 to ensure $h_{\ell-1}\ge \varepsilon^2$.  As a result, $h_{\ell-1}^2 \varepsilon^2$ is 
likely to be the dominant  term in \eqref{eq:main}. See Figure \ref{fig:example_simple2}(b), where the log-log  plot is consistent with the functional form
\begin{align*}
			\textsf{Var}(D_{h_{\ell}}^\varepsilon(1) - D_{h_{\ell-1}}^\varepsilon(1)) = O(\varepsilon^2).
\end{align*}

\begin{figure}
\centering
\subfigure[$h_{\ell-1}= 2^{-19}$ fixed while $\varepsilon$ is varied. The best fit curve is $y=3.96x-16.68$.]{\includegraphics[width=0.47\textwidth]{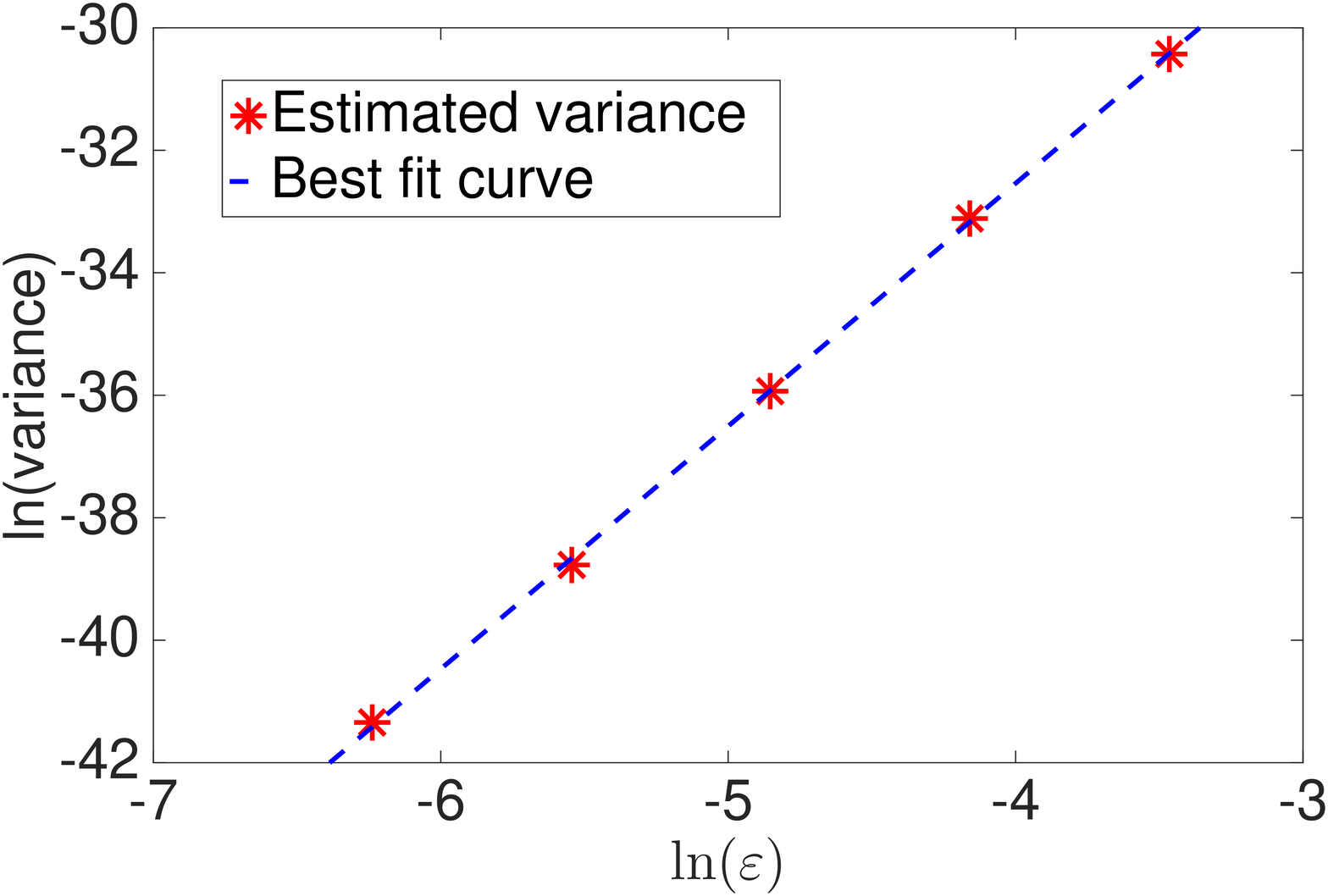}}
\hspace{.2in}
\subfigure[$h_{\ell-1}= 2^{-9}$ fixed while $\varepsilon$ is varied. The best fit curve is $y=2.08x-16.7$.]{\includegraphics[width=0.47\textwidth]{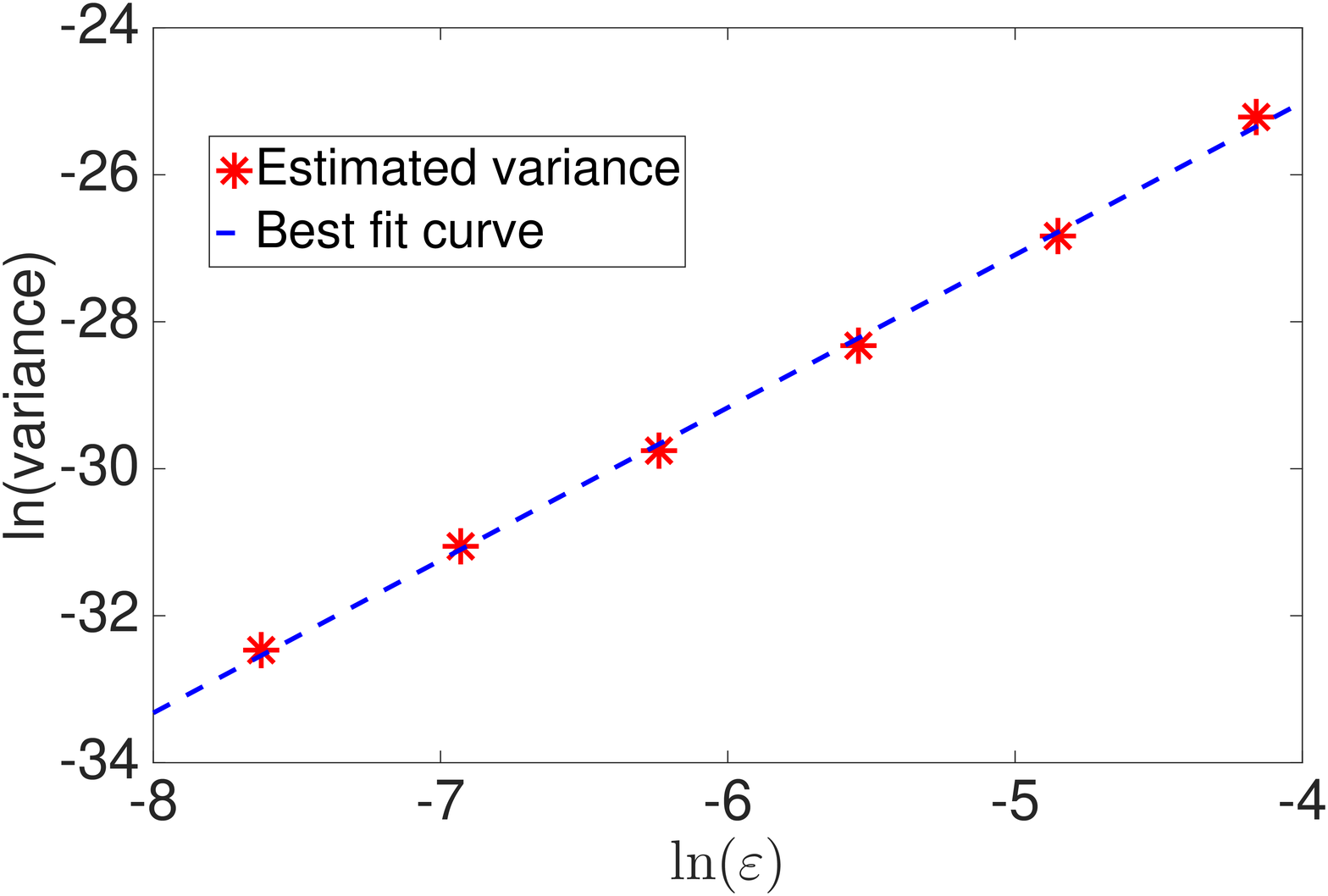}}
	\caption{Log-log plots of $\textsf{Var}(D_{h_{\ell}}^\varepsilon(1) - D_{h_{\ell-1}}^\varepsilon(1))$ with $h_{\ell-1}$ held constant and $\varepsilon$  varied. The best fit curves for all data are overlain in the dashed blue line.  Each data point was generated using 1,000 independent samples.}
	\label{fig:example_simple2}
\end{figure}

We turn to numerically demonstrating our conclusions related to the complexity of Euler based multilevel Monte Carlo and the complexity of Euler based standard Monte Carlo.   We will measure complexity in two ways, by total number of random variables utilized and by required CPU time.   Our implementation of MLMC proceeded as follows.  We chose $h_\ell = 2^{-\ell}$ and for each $\delta>0$ we set $L = \lceil\log(\delta)/\log(2)\rceil$.  
For each level we generated 200 independent sample trajectories in order to estimate $\delta_{\varepsilon, \ell}$, as defined in section \ref{subsec:emmc}.  According to \eqref{eq:nl} and \eqref{eq:09877890} we then selected  
\[
n_{\ell}=\Bigg\lceil\delta^{-2} \sqrt{ \delta_{\varepsilon,\ell} h_{\ell}}\sum_{j = 0}^L\sqrt{ \frac{\delta_{\varepsilon,j}}{h_j}}\Bigg\rceil, \qquad \text{ for } \ell \in\{0,1,2,\dots,L\}.
\]
We implemented Euler's method combined with standard Monte Carlo by selecting the number of paths by
\[
N=\left\lceil\delta^{-2}\textsf{Var}(D_h^\varepsilon(1))\right\rceil
\]
where $h=2^{-L}$ and the parameter $\textsf{Var}(D_h^\varepsilon(1))$ was estimated using 500 independent realizations of the relevant processes.

In  Figures \ref{fig:runtime}(a) and  \ref{fig:random}(a),  we provide log-log plots of runtime (in seconds) and complexity (quantified by the total number of random variables utilized) for our implementation of multilevel and standard Monte Carlo with  $\varepsilon=0.1$ fixed and  
\[
\delta \in \{0.00032, 0.00016, 0.00008, 0.00004, 0.00002\},
\]
which ensures $\delta>\frac{1}{3}e^{-\frac{1}{\varepsilon}}$ (see section \ref{subsec:emmc}). The best fit curves   are consistent with the conclusion that  the computational complexity of the Euler based multilevel Monte Carlo method is $O(\delta^{-2})$ while that of standard Monte Carlo method is $O(\delta^{-3})$ when $\varepsilon$ is fixed. The Monte Carlo estimates which came from these simulations are detailed in Table \ref{table:varydeltaMLMCMasterTable}.
Notice  that $\E [D^\varepsilon(1)] $ can be found explicitly in this case, 
\[
\E [D^\varepsilon(1)]=e^{-1}\approx 0.3678794.
\]
\begin{figure}
\centering
\subfigure[$\varepsilon= 0.1$ held constant and $\delta$  varied.  The best fit lines are  $y=-1.96x-15.65$ for Euler based MLMC
and $y=-2.91x -5.44$ for standard Euler based Monte Carlo.]{\includegraphics[width=0.47\textwidth]{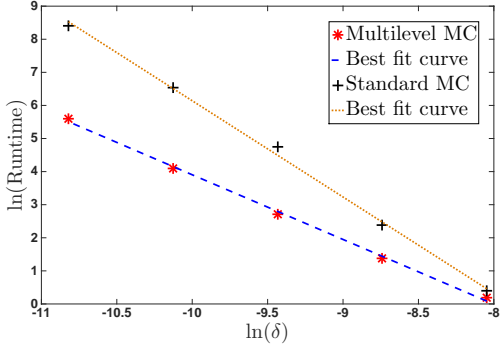}}
\hspace{.2in}
\subfigure[$\delta= 2^{-14}$ held constant and $\varepsilon$  varied.  The best fit lines are $y=2.07x+ 9.18$ for Euler based MLMC
and $y=1.99x+9.58$ for standard Euler based Monte Carlo.]{\includegraphics[width=0.47\textwidth]{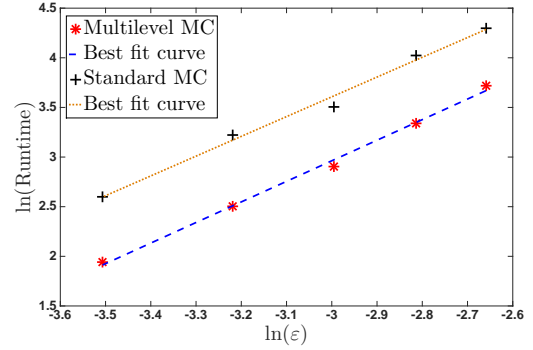}}
	\caption{Log-log plots of runtime (in seconds) for both multilevel and standard Euler based Monte Carlo.}
	\label{fig:runtime}
\end{figure}

\begin{figure}
\centering
\subfigure[$\varepsilon= 0.1$ held constant and $\delta$  varied. The best fit lines are  $y=-2.04x -3.33$ for MLMC
and $y=-2.91x -22.95$ for standard MC.]{\includegraphics[width=0.47\textwidth]{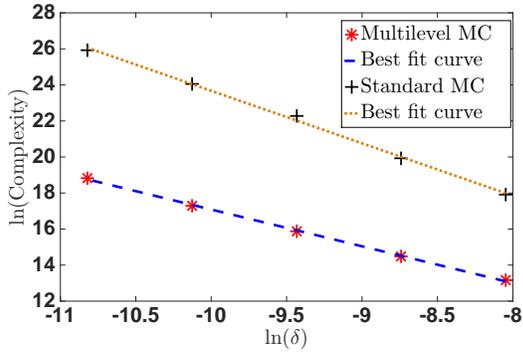}}
\hspace{.2in}
\subfigure[$\delta= 2^{-14}$ held constant and $\varepsilon$  varied.  The best fit lines  are $y=2.06x+  21.13$ for MLMC
and $y=1.99x+27.12$ for standard MC.]{\includegraphics[width=0.47\textwidth]{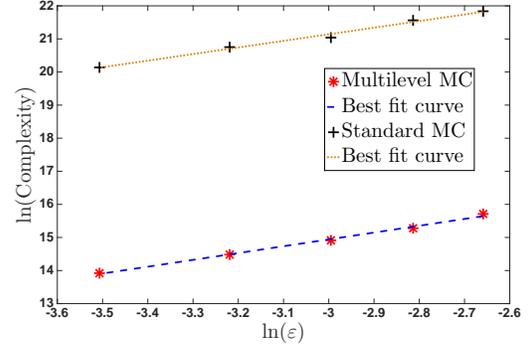}}
	\caption{Log-log plots of computational complexity,  quantified by the number of random variables used.}
\label{fig:random}
\end{figure}

\begin{table}[b]
\centering
  \begin{tabular}{ |c|c|c|c|c|}
    \hline
    $\delta$& Mean--Euler & Mean--MLMC  & SD--Euler & SD--MLMC \\ \hline
    0.00032 & 0.367449 & 0.367944 &$0.000320$  & $ 0.000305$\\ \hline
    0.00016 &0.368028 & 0.367906  & $0.000160$ & $0.000153$ \\ \hline
    0.00008 & 0.367839& 0.367891 &$0.000080$  & $0.000077$\\ \hline
    0.00004 &0.367941 & 0.367863 & $0.000040$& $0.000039$ \\ \hline
    0.00002 &0.367851 & 0.367883 & $0.000020$ & $0.000020$ \\ \hline
  \end{tabular}
  \caption{Result of Euler based multilevel Monte Carlo and Euler based Monte Carlo for fixed  $\varepsilon=0.1$ and varying $\delta$.  The last two columns provide the standard deviations for the two estimators.}
  \label{table:varydeltaMLMCMasterTable}
\end{table}

%
%
%

In  Figure \ref{fig:runtime}(b) and  \ref{fig:random}(b),  we provide similar log-log plots of runtime and computational complexity for Euler based multilevel Monte Carlo and standard Monte Carlo when $\delta=2^{-14}$ is fixed and $\varepsilon$ is varied as 
\[
\varepsilon\in \{0.07, 0.06, 0.05, 0.04, 0.03\},
\]
which ensures $\delta>e^{-\frac{1}{\varepsilon}}$. The best fit curves are again consistent with the conclusion that  the complexity of Euler based multilevel Monte Carlo and standard Monte Carlo Methods  are  both $O(\varepsilon^{2})$ when $\delta$ is fixed. The Monte Carlo estimates which came from these simulations are detailed in Table  \ref{table:varyepsilonMLMCmaster}. \hfill $\square$

\begin{table}
\centering
  \begin{tabular}{ |c|c|c|c|c|}
    \hline
    $\varepsilon$& Mean-Euler & Mean-MLMC & SD-Euler & SD-MLMC \\ \hline
    0.07 & 0.367830&  0.367834 &  $0.000061$  & $0.000059$\\ \hline
    0.06 & 0.367755&  0.367920 & $0.000061$ &  $0.000059$ \\ \hline
    0.05 & 0.367933& 0.367819 &  $0.000061$ & $0.000059$\\ \hline
    0.04 & 0.367809& 0.367856 &  $0.000061$& $0.000059$ \\ \hline
    0.03 & 0.367879& 0.367925 &  $0.000061$ & $0.000059$ \\ \hline
  \end{tabular}
  \caption{Results of Euler based multilevel Monte Carlo and Euler based Monte Carlo  for fixed $\delta=2^{-14}\approx 0.000061$ and varying $\varepsilon$.  The last two columns provide the standard deviations for the two estimators.}
  \label{table:varyepsilonMLMCmaster}
\end{table}

%

\end{example}

\subsection{Comparison with results for continuous time Markov chains}
\label{sec:comparisonCTMC}

Diffusion processes with small noise structures of the form \eqref{eq:2408957248907} often arise as  approximations to  continuous time Markov chains.  Bounds on the variance between Euler approximations to such scaled jump processes can be found in  \cite{AHS2014}.  Since the diffusion approximation is naturally related to the jump process model, it is tempting to believe that the analysis found in \cite{AHS2014} can be utilized to infer the results presented in this paper.  The following example and analysis shows that this intuition is incorrect 

\begin{example}\label{example2}
Consider a family of  continuous time Markov chain models,  parametrized by $N>0$, satisfying
\begin{align}
\begin{split}
	X^N(t) = X^N(0) + &\frac1N Y_1\left( N \int_0^t X_1^N(s)(X_1^N(s)-\tfrac1N) ds\right)\left[ \begin{array}{c} -2\\ 1 \end{array} \right]\\[1ex]
	\hspace{.125in}  + &\frac1N Y_2\left( N \int_0^t X_2^N(s)ds\right)\left[ \begin{array}{c} 2\\ -1 \end{array} \right]
	\end{split}
	\label{eq:9879786}
\end{align} 
with $X^N(0) \in \tfrac1N \Z^2_{\ge 0}$ and $Y_1,Y_2$ independent unit-rate Poisson processes.  The process \eqref{eq:9879786} can model the time evolution of the \textit{reaction network}
\begin{equation}
	2A \overset{1/N}{\underset{1}{\rightleftarrows}} B,
\end{equation}
in which two molecules of species $A$ can combine to form a molecule of species $B$, and vice versa \cite{AndKurtz2011,AK2015}.    The specific choice of scaling in \eqref{eq:9879786} is called the \textit{classical scaling} for biochemical processes \cite{AndKurtz2011,AK2015}.   One representation for the continuous in time Euler-Maruyama approximation of the standard diffusion approximation to the model \eqref{eq:9879786}  is 
\begin{align}
\label{eq:5686875}
\begin{split}
	D_h^{N}(t) &= D_h^N(0) + \int_0^t D_{h,1}^N(\eta_h(s))^2 ds \left[ \begin{array}{c} -2\\ 1 \end{array} \right] + \int_0^t D_{h,2}^N(\eta_h(s)) ds \left[ \begin{array}{c} 2\\ -1 \end{array} \right] \\
	&\hspace{.2in} +\varepsilon_N \int_0^t \sqrt{\max\{D_1^N(\eta_h(s))^2, 0\}} \, dW_1(s) \left[ \begin{array}{c} -2\\ 1 \end{array} \right]\\
	&\hspace{.2in}+ \varepsilon_N \int_0^t \sqrt{\max\{D_2^N(\eta_h(s)), 0\}} \, dW_2(s) \left[ \begin{array}{c} 2\\ -1 \end{array} \right],
	\end{split}
\end{align}
where $\varepsilon_N = \tfrac1{\sqrt{N}}$, $W_1$ and $W_2$ are independent Brownian motions \cite{AndKurtz2011,AK2015}.  
Let $Z_h^N$ be an Euler approximation to \eqref{eq:9879786} and
let $M>0$ be some fixed positive integer.  See \cite{AndersonGangulyKurtz} or \cite{AHS2014} for a stochastic representation of $Z_h^N$ that is similar to \eqref{eq:5686875}, and for the relevant coupling between $Z_{h_\ell}^N$ and $Z_{h_{\ell-1}}^N$.  By Corollary 1 in \cite{AHS2014}, we have that for $h_\ell = M^{-\ell}$ 
\begin{align}\label{eq:7687968769}
	\text{Var}(Z_{h_\ell,1}^N(t) - Z_{h_{\ell-1},1}^N(t)) \le D\cdot  h_{\ell-1} N^{-1} = D\cdot  h_{\ell-1}\varepsilon_N^2,
\end{align}
where the constant $D$  does not depend upon $N, h_\ell$ , or $h_{\ell-1}$.    Conversely, Theorem \ref{thm:var} allows us to conclude that
\begin{align}\label{eq:11111111}
	\text{Var}(D_{h_\ell,1}^N(t) - D_{h_{\ell-1},1}^N(t) ) \le C_1 \cdot h_{\ell-1}^2 \varepsilon_N^2 + C_2\cdot h_{\ell-1}\varepsilon_N^4.
\end{align}
The key feature to note is that for $h_{\ell-1}<1$ and $\varepsilon_N<1$, both of the terms $h_{\ell-1}^2 \varepsilon_N^2$ and $h_{\ell-1}\varepsilon_N^4$ are  dominated by $h_{\ell-1}\varepsilon_N^2$.  In fact,
\[
	\frac{h_{\ell-1}^2 \varepsilon_N^2 + h_{\ell-1}\varepsilon_N^4}{h_{\ell-1}\varepsilon_N^2} = h_{\ell-1} + \varepsilon_N^2,
\]
showing a dramatic reduction in the variance when the coupled diffusion processes are considered as opposed to  the coupled jump processes. 

In order to numerically demonstrate the bounds \eqref{eq:7687968769} and \eqref{eq:11111111}, we follow the numerical analysis performed in Example \ref{example1} by varying $h_\ell$ and $\varepsilon_N = \tfrac{1}{\sqrt{N}}$ in different scaling regimes in order to isolate the different possible exponents.  For each of the numerical experiments performed we fixed a terminal time of $T=0.3$ and took an initial condition of  $(0.2, 0.2)$ for each model.  As we also mentioned in Example \ref{example1}, we emphasize that these experiments use extreme parameter choices solely for the purpose
of testing the sharpness of the delicate asymptotic bounds.

\vspace{.1in}

\noindent \textbf{The exponent of $h_{\ell-1}$ in $h_{\ell-1}\varepsilon_N^4$.} 
We fix $N= 2^{12},$ which corresponds with $\varepsilon_N = 2^{-6}$, and vary     
\[
	h_{\ell-1}\in\{2^{-13},2^{-14},2^{-15},2^{-16}\}
\]
 to ensure $h_{\ell-1}\le \varepsilon_N^2$.  As a result, $h_{\ell-1}\varepsilon_N^4$ is likely to be the dominant  term in \eqref{eq:11111111}. See Figure \ref{fig:example_compare1}(a), where the log-log  plots are consistent with the functional forms
\begin{align*}
			\textsf{Var}(D_{h_{\ell},1}^N(T) - D_{h_{\ell-1},1}^N(T)) = O(h_{\ell-1}), \quad \textsf{Var}(Z_{h_{\ell},1}^N(T) - Z_{h_{\ell-1},1}^N(T)) = O(h_{\ell-1}).
\end{align*}

\begin{figure}
\centering
\subfigure[$\varepsilon_N = 2^{-6}$ fixed while $h_{\ell-1}$ is varied.  The best fit curve for the data associated with the CTMC is $y=1.07x-9.21$, whereas the best fit curve for the data associated with the diffusion  is $y=1.04x-16.92$.]{\includegraphics[width=0.47\textwidth]{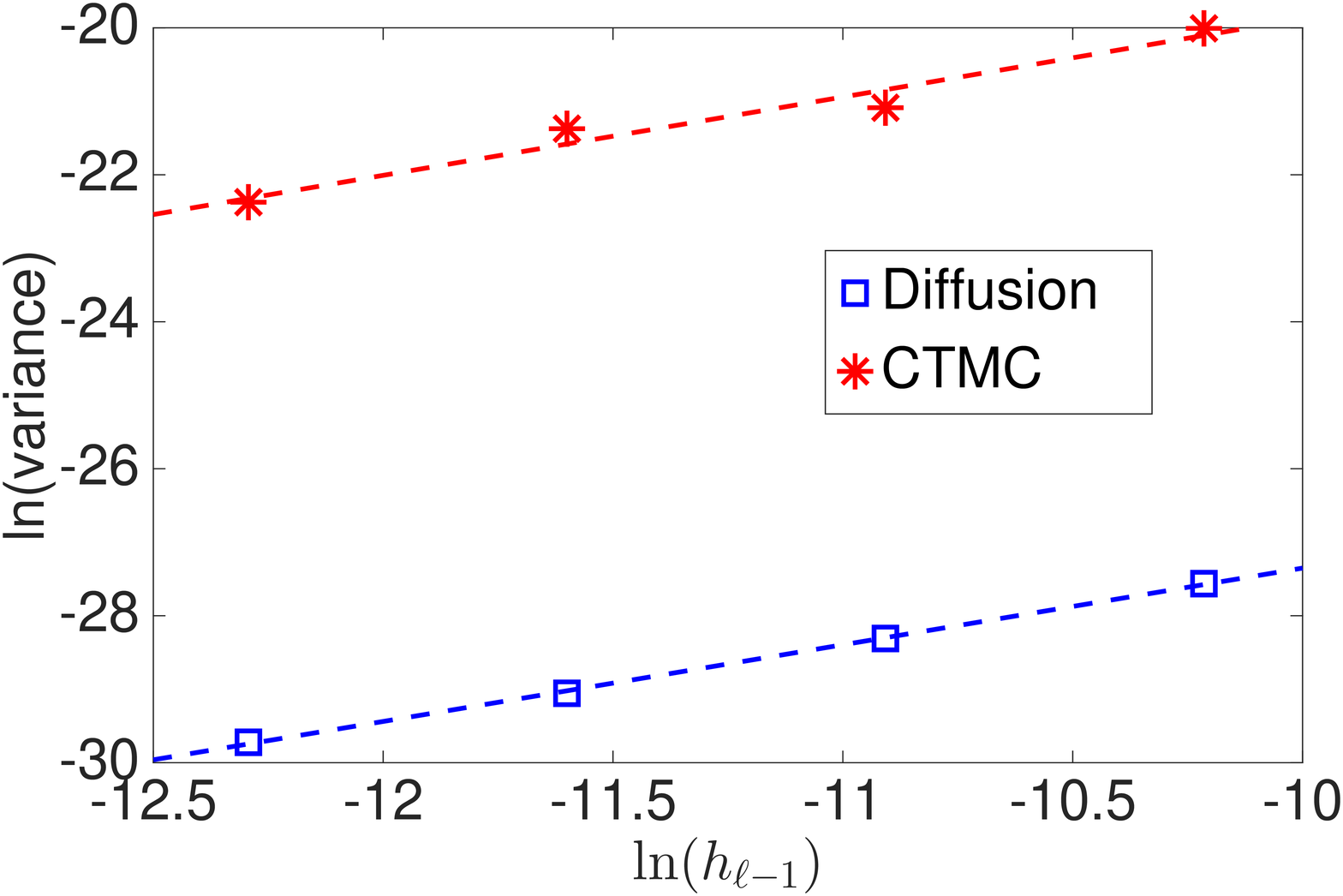}}
\hspace{.2in}
\subfigure[$\varepsilon_N = 2^{-10}$ fixed while $h_{\ell-1}$ is varied.  The best fit curve for the data associated with the CTMC is $y=0.95x-16.82$, whereas the best fit curve for the data associated with the diffusion  is $y=1.98x-15.25$.]{\includegraphics[width=0.47\textwidth]{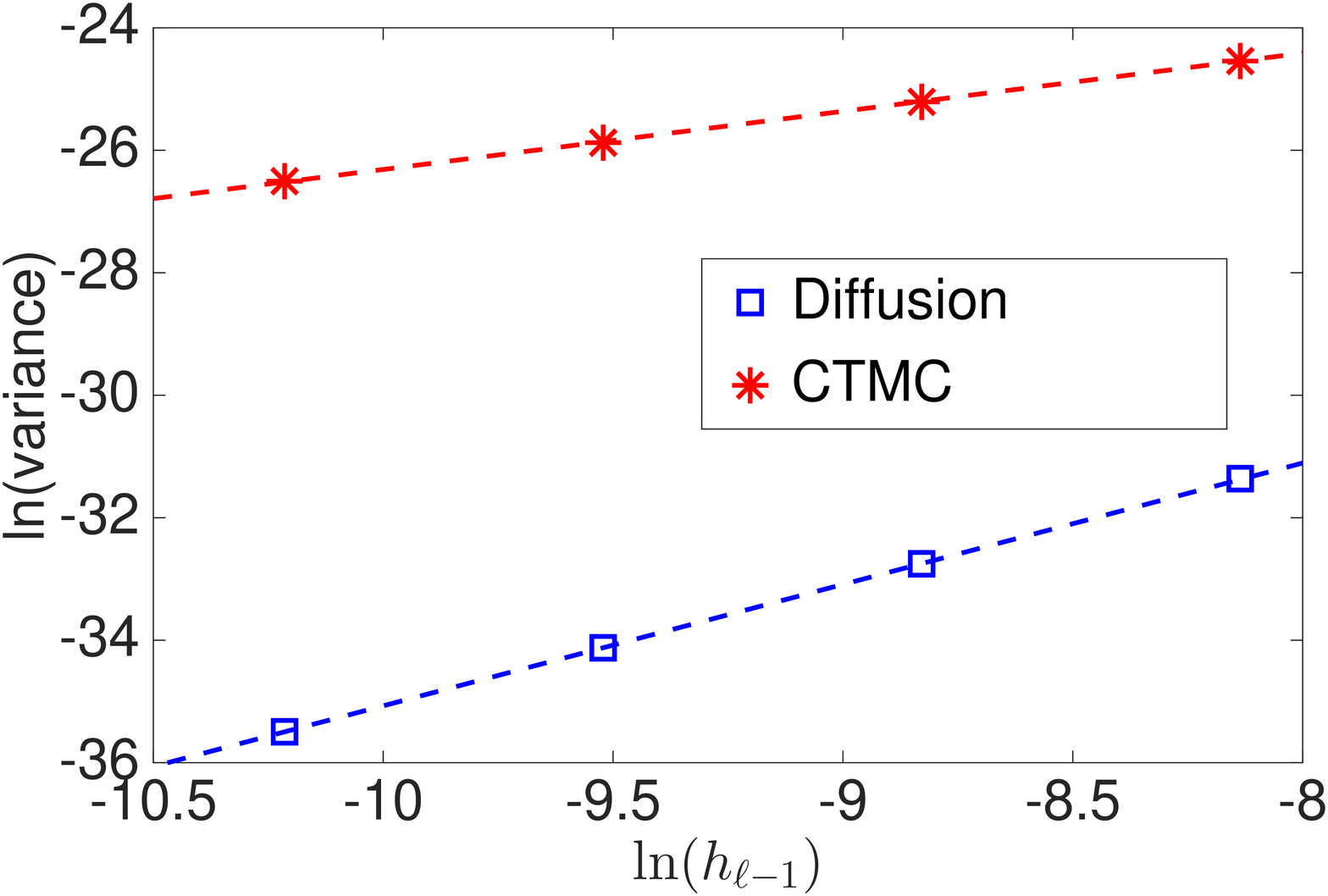}}
	\caption{Log-log plots of $\textsf{Var}(D_{h_{\ell},1}^N(T) - D_{h_{\ell-1},1}^N(T))$ and $\textsf{Var}(Z_{h_{\ell},1}^N(T) - Z_{h_{\ell-1},1}^N(T))$ with $\varepsilon_N$  held constant and $h_{\ell-1}$ varied. The best fit curves for  the data are overlain with  dashed  lines. }
\label{fig:example_compare1}
\end{figure}

\vspace{.1in}

\noindent \textbf{The exponent of $h$ in $h^2\varepsilon^2$.} 
We fix $N=2^{20}$, which corresponds with $\varepsilon_N = 2^{-10}$ and vary   
\[
	h_{\ell-1}\in\{2^{-10},2^{-11},2^{-12},2^{-13}\}
	\]
	 to ensure $h_{\ell-1}\ge \varepsilon_N^2$.  As a result, $h_{\ell-1}^2 \varepsilon_N^2$ is likely to be the dominant  term in \eqref{eq:11111111}. See Figure \ref{fig:example_compare1}(b), where the log-log  plots are consistent with the functional forms
\begin{align*}
			\textsf{Var}(D_{h_{\ell},1}^N(T) - D_{h_{\ell-1},1}^N(T)) = O(h_{\ell-1}^2), \quad \textsf{Var}(Z_{h_{\ell},1}^N(T) - Z_{h_{\ell-1},1}^N(T)) = O(h_{\ell-1}).
\end{align*}

\vspace{.1in}

\noindent \textbf{The exponent of $\varepsilon$ in $h\varepsilon^4$.} 
  We fix $h_{\ell-1}= 2^{-12}$  and vary   
\[
	N^{-1} \in \{2^{-6},2^{-7},2^{-8},2^{-9},2^{-10},2^{-11}\}
	\]
	 to ensure $h_{\ell-1}\le  \varepsilon_N^2=N^{-1}$.  As a result, $h_{\ell-1} \varepsilon_N^4$ is
likely to be  the dominant  term in \eqref{eq:11111111}. See Figure \ref{fig:example_compare2}(a), where the log-log  plot is consistent with the functional form
\begin{align*}
			\textsf{Var}(D_{h_{\ell},1}^N(T) - D_{h_{\ell-1},1}^N(T)) = O(\varepsilon_N^4), \quad \textsf{Var}(Z_{h_{\ell},1}^N(T) - Z_{h_{\ell-1},1}^N(T)) = O(\varepsilon_N^2)
\end{align*}

\vspace{.1in}

\noindent \textbf{The exponent of $\varepsilon$ in $h^2\varepsilon^2$.} 
 We fix $h_{\ell-1}= 2^{-11}$  and vary   
\[
	N \in \{2^{20},2^{21},2^{22},2^{23},2^{24},2^{25}\}
	\]
	 to ensure $h_{\ell-1}\ge \varepsilon_N^2= N^{-1}$.  As a result, $h_{\ell-1}^2 \varepsilon_N^2$ is 
likely to be the dominant  term in \eqref{eq:11111111}. See Figure \ref{fig:example_compare2}(b), where the log-log  plot is consistent with the functional form
\begin{align*}
			\textsf{Var}(D_{h_{\ell},1}^N(T) - D_{h_{\ell-1},1}^N(T)) = O(\varepsilon_N^2),  \quad \textsf{Var}(Z_{h_{\ell},1}^N(T) - Z_{h_{\ell-1},1}^N(T)) = O(\varepsilon_N^2).
\end{align*}

\begin{figure}
\centering
\subfigure[$h_{\ell-1}= 2^{-12}$ fixed while $\varepsilon_N$ is varied. The best fit curve for the data associated with the CTMC is $y=2.14x-11.07$, whereas the best fit curve for the data associated with the diffusion is $y=3.99x-11.02$.]{\includegraphics[width=0.47\textwidth]{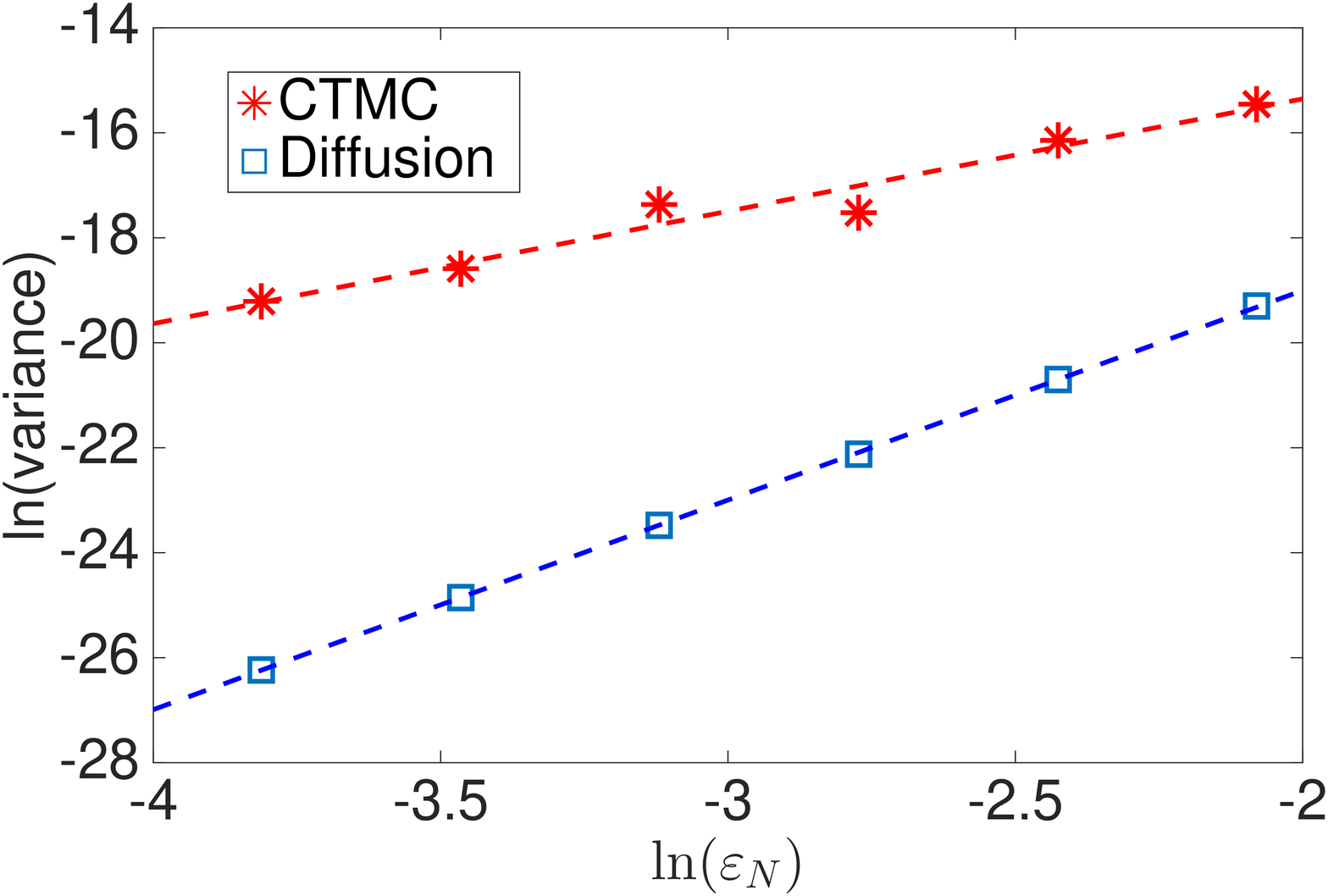}}
\hspace{.2in}
\subfigure[$h_{\ell-1}= 2^{-11}$ fixed while $\varepsilon_N$ is varied. The best fit curve for the data associated with the CTMC is $y=1.99-12.11$, whereas the best fit curve for the data associated with the diffusion is $y=2.00x-20.25$.]{\includegraphics[width=0.47\textwidth]{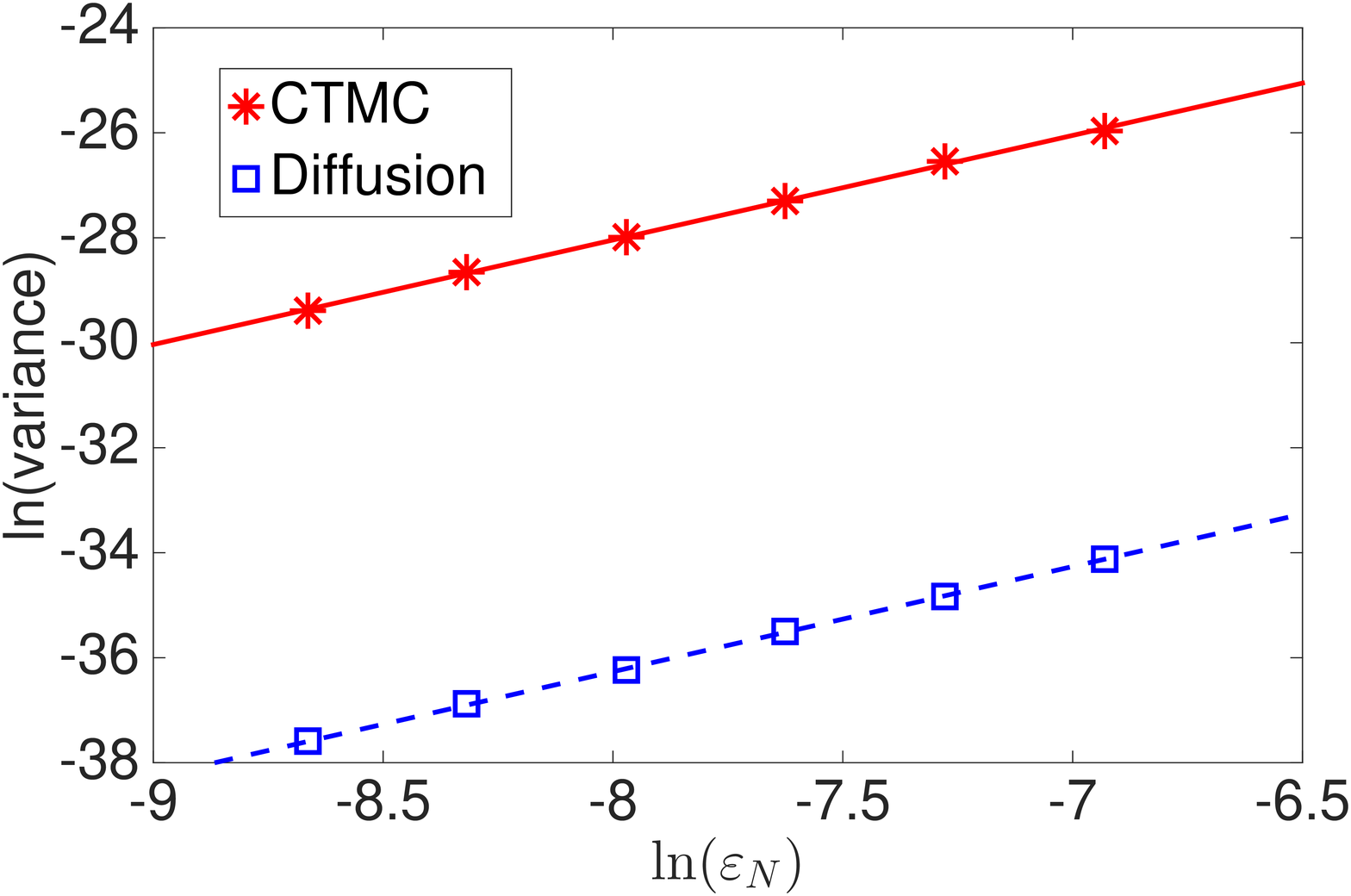}}
	\caption{Log-log plots of $\textsf{Var}(D_{h_{\ell},1}^N(T) - D_{h_{\ell-1},1}^N(T))$ and $\textsf{Var}(Z_{h_{\ell},1}^N(T) - Z_{h_{\ell-1},1}^N(T))$ with $h_{\ell-1}$ held constant and $\varepsilon_N$  varied. The best fit curves for all data are overlain with dashed lines.}
	\label{fig:example_compare2}
\end{figure}

\end{example}

\section{Summary}
\label{sec:summary}

This work focussed on Monte Carlo methods for approximating expectations 
arising from SDEs with small noise.
Our motivation was that for the highly effective 
multilevel approach, the classical strong error measure is less relevant than the variance   
between coupled pairs of paths at different discretization levels.
By analyzing this variance directly,
we showed that when $\delta \le \varepsilon^2$ there is no benefit from using discretization methods that are customized for small noise.  Moreover, so long as we also have $\delta \ge e^{-\frac1\varepsilon}$, a basic Euler--Maruyama discretization 
used in a multilevel setting leads to the same complexity that would arise 
 in the idealized case where we had access
 to exact samples of the
required distribution at a cost of $O(1)$ per sample.  

Interesting future work in this area includes the following.
\begin{enumerate}[(i)]
\item Develop multilevel methods customized to the setting $\delta > \varepsilon^2$.  

\item Investigate whether the recently proposed techniques in \cite{BN14} and \cite{MGY14} can be adapted to the small noise regime.
\end{enumerate}

\appendix
\section{Some Technical Lemmas}
\label{sec:app}

We provide here some technical lemmas which were used in section \ref{sec:res}. 

The following is Lemma 5 in  the appendix of \cite{AHS2014}. 

\begin{lemma}\label{lem:twovar_bound}
 \label{lem:orderv}
Suppose $X_1(s)$ and $X_2(s)$ are  stochastic processes on $\R^d$ and that   $x_1(s)$ and $x_2(s)$ are  deterministic processes on $\R^d$.  Further, suppose that
 \begin{align*}
        \sup_{s \le T} \E\left[ |X_1(s)-x_1(s)|^2\right] \le \widehat{C}_1(T)\varepsilon^2,\quad  \sup_{s \le T} \E \left[ |X_2(s)-x_2(s)|^2\right] \le \widehat{C}_2(T)\varepsilon^2,
 \end{align*}
 for some $\widehat{C}_1,\widehat{C}_2$ depending upon $T$.
   Assume that $u: \R^d \to \R$ is Lipschitz  with Lipschitz constant $C_L$. Then,
		 \[
     \sup_{s \le T}\textsf{Var} \left( \int_0^1  u(X_2(s) + r(X_1(s)-X_2(s)))dr\right)\le C_L^2 \max(\widehat{C}_1,\widehat{C}_2)\varepsilon^2.
		 \]
 \end{lemma}

The following lemma is only a slight perturbation of Lemma 6 in \cite{AHS2014}.  A proof is
therefore  omitted.
\begin{lemma}\label{lem:Var_bound}
	Suppose that $A^{\varepsilon,h}$ and $B^{\varepsilon,h}$ are families of random variables determined by  scaling parameters $\varepsilon$ and $h$.  Further, suppose that there are $C_1>0, C_2>0$ and $C_3>0$ such that  for all $\varepsilon\in(0,1)$ the following three conditions hold:
	\begin{enumerate}
		\item $\textsf{Var}(A^{\varepsilon,h}) \le C_1\varepsilon^2$ uniformly in $h$.
		\item $|A^{\varepsilon,h}|\le C_2$ uniformly in $h$.
		\item $|\E [B^{\varepsilon,h}]| \le C_3h$.
	\end{enumerate}
	 Then
	\begin{equation*}
		\textsf{Var}(A^{\varepsilon,h} B^{\varepsilon,h}) \le 3C_3^2C_1h^2\varepsilon^2+15C_2^2\textsf{Var}(B^{\varepsilon,h}).
	\end{equation*}
\end{lemma}

The following lemma is standard, but is included for completeness.
 \begin{lemma}\label{lem:taylor}
 	Let $f:\R^d \to \R$ have continuous first derivative.  Then, for any $x,y\in \R^d$,
	\[
		f(x) = f(y) + \int_0^1 \nabla f(sx + (1-s)y) ds \cdot (x-y).
	\]
 \end{lemma}

\bibliographystyle{siam}
\bibliography{SmallNoiseMLMC,SmallNoiserefs}

\end{document}